\documentclass[11pt]{amsart}
\usepackage[dvipsnames,svgnames,x11names]{xcolor}
\usepackage{textgreek,bbm,url,graphicx,verbatim,amssymb,enumerate,stmaryrd,booktabs,mathtools,microtype,mathabx,}

\usepackage{amsmath, amsthm, amssymb}
\usepackage{amsfonts}
\usepackage[pagebackref,colorlinks,citecolor=Mahogany,linkcolor=Mahogany,urlcolor=Mahogany,filecolor=Mahogany]{hyperref}
\usepackage[capitalize]{cleveref}
\usepackage{float} 
\usepackage[mathscr]{euscript}
\usepackage{tikz,tikz-cd}
\usetikzlibrary{matrix,calc,positioning,arrows,decorations.pathreplacing,patterns,arrows}
\usepackage{mathabx}
\usepackage{fancyhdr}
\usepackage{parskip}
\usepackage[a4paper,
         bindingoffset=0.2in,
         left=0.8in,
         right=1in,
         top=1in,
         bottom=1in,
         footskip=.25in,
         headheight=15pt]{geometry}
\usepackage{enumitem}
\setlist[enumerate,1]{label=\textup{(\arabic*)}}

\makeatletter
\DeclareFontFamily{OMX}{MnSymbolE}{}
\DeclareSymbolFont{MnLargeSymbols}{OMX}{MnSymbolE}{m}{n}
\SetSymbolFont{MnLargeSymbols}{bold}{OMX}{MnSymbolE}{b}{n}
\DeclareFontShape{OMX}{MnSymbolE}{m}{n}{
    <-6>  MnSymbolE5
   <6-7>  MnSymbolE6
   <7-8>  MnSymbolE7
   <8-9>  MnSymbolE8
   <9-10> MnSymbolE9
  <10-12> MnSymbolE10
  <12->   MnSymbolE12
}{}
\DeclareFontShape{OMX}{MnSymbolE}{b}{n}{
    <-6>  MnSymbolE-Bold5
   <6-7>  MnSymbolE-Bold6
   <7-8>  MnSymbolE-Bold7
   <8-9>  MnSymbolE-Bold8
   <9-10> MnSymbolE-Bold9
  <10-12> MnSymbolE-Bold10
  <12->   MnSymbolE-Bold12
}{}

\let\llangle\@undefined
\let\rrangle\@undefined
\DeclareMathDelimiter{\llangle}{\mathopen}%
                     {MnLargeSymbols}{'164}{MnLargeSymbols}{'164}
\DeclareMathDelimiter{\rrangle}{\mathclose}%
                     {MnLargeSymbols}{'171}{MnLargeSymbols}{'171}
\makeatother

\def\XXint#1#2#3{{\setbox0=\hbox{$#1{#2#3}{\int}$ }
\vcenter{\hbox{$#2#3$ }}\kern-.6\wd0}}

\newtheorem{theorem}{Theorem}[section]
\newtheorem{lemma}[theorem]{Lemma}
\newtheorem{proposition}[theorem]{Proposition}
\newtheorem{corollary}[theorem]{Corollary}
\newtheorem*{corollary*}{Corollary}

\newtheorem{atheorem}{Theorem}

\newtheorem{acorollary}[atheorem]{Corollary}

\theoremstyle{definition}
\newtheorem{definition}[theorem]{Definition}

\theoremstyle{remark}

\newtheorem{remark}[theorem]{Remark}

\newcommand{\subref}[2]{\hyperref[#2]{\ref{#1}.\ref{#2}}}

\numberwithin{equation}{section}
\crefname{equation}{}{}

\newcommand{\sq}[1]{\widetilde{#1}}
\newcommand{\R}{\mathbb{R}}
\newcommand{\Z}{\mathbb{Z}}
\newcommand{\N}{\mathbb{N}}

\newcommand{\mc}[1]{\mathcal{#1}}

\newcommand{\ang}[1]{\langle #1 \rangle}
\newcommand{\dang}[1]{\llangle #1 \rrangle}
\newcommand{\D}{\mathcal{D}}

\newcommand{\1}{\mathbf{1}}
\newcommand{\LL}{\langle}
\newcommand{\RR}{\rangle}
\newcommand{\BMO}{\mathrm{BMO}}
\newcommand{\bmo}{\mathrm{bmo}}

\newcommand{\Hilb}{\mathcal{H}}
\newcommand{\E}{\mathbb{E}}

\definecolor{darkgreen}{HTML}{005239}

\author{Francesco D'Emilio}
\address{Francesco D'Emilio\hfill\break\indent 
 Department of Mathematics \hfill\break\indent 
Washington University in St. Louis \hfill\break\indent 
 One Brookings Drive \hfill\break\indent 
St. Louis, MO 63130 USA}
\email{demilio@wustl.edu}

\author{Yongxi Lin}
\address{Yongxi Lin \hfill\break\indent 
 Department of Mathematical Sciences \hfill\break\indent 
Carnegie Mellon University \hfill\break\indent 
5000 Forbes Ave
  \hfill\break\indent 
Pittsburgh, PA 15213 USA}
\email{aaronlin@andrew.cmu.edu}

\author{Nathan A. Wagner} 
\address{Nathan A. Wagner\hfill\break\indent 
 Department of Mathematical Sciences \hfill\break\indent 
George Mason University \hfill\break\indent 
4400 University Dr
  \hfill\break\indent 
Fairfax, VA 22030 USA}
\email{nwagner8@gmu.edu}

\author{Brett D. Wick}
\address{Brett D. Wick \hfill\break\indent 
 Department of Mathematics \hfill\break\indent 
Washington University in St. Louis \hfill\break\indent 
 One Brookings Drive \hfill\break\indent 
St. Louis, MO 63130 USA}
\email{bwick@wustl.edu}

\title{Optimal Sparse Bounds and Commutator Characterizations Without Doubling}

\thanks{\begin{footnotesize}Nathan A. Wagner was supported by National Science Foundation grants DMS 203272 and 2549719.\end{footnotesize}}

\thanks{\begin{footnotesize}
Brett D. Wick was partially supported by National Science Foundation DMS 2349868.
\end{footnotesize}}

\begin{document}
\begin{abstract}
We examine dyadic paraproducts and commutators in the non-homogeneous setting, where the underlying Borel measure $\mu$ is not assumed to be doubling. We first establish a pointwise sparse domination for dyadic paraproducts and related operators with symbols $b \in \BMO(\mu)$, improving upon an earlier result of Lacey, where the symbol $b$ was assumed to satisfy a stronger Carleson-type condition, that coincides with $\BMO$ only in the doubling setting. As an application of this result, we obtain sharpened weighted inequalities for the commutator of a dyadic Hilbert transform $\mathcal{H}$ previously studied by Borges, Conde Alonso, Pipher, and the third author. We also characterize the symbols for which the commutator $[\mathcal{H},b]$  is bounded on $L^p(\mu)$ for $1<p<\infty$ and provide some interesting examples to prove that this class of symbols strictly depends on $p$ and is nested between symbols satisfying the $p$-Carleson packing condition and symbols belonging to martingale BMO (even in the case of absolutely continuous measures).
    
\end{abstract}

\maketitle

\section{Introduction}

The theory of commutators in harmonic analysis presents a striking dichotomy: while completely understood in homogeneous settings through the classical BMO characterization of Coifman, Rochberg, and Weiss \cite{CRW1976}, these operators can exhibit a fundamentally different behavior when the underlying measure lacks the doubling property. This breakdown is not a mere technical inconvenience: in the nonhomogeneous setting there appears to be a fundamental bifurcation between continuous and dyadic Calder\'{o}n-Zygmund models, breaking a connection that proved to be immensely powerful and fruitful in the doubling case. This reveals that our standard tools, from dyadic decompositions to sparse domination, require fundamental reconsideration. 
 \par 

Recent progress in nonhomogeneous dyadic theory builds upon the pioneering works \cite{TREIL2010}, \cite{LSMP}, and \cite{Lacey_2017}, where the authors developed the unweighted theory for martingale transforms, Haar shifts, paraproducts with martingale BMO symbols, and commutators with martingale transforms. However, classical results are not always recovered as seamlessly as one might expect; additional structural assumptions are often required to obtain meaningful answers. Despite powerful advances in the weighted theory in the recent years \cite{CPW, BCAPW, BBDPW}, basic questions remain unresolved: \begin{enumerate}
    \item To what extent can sparse domination be extended beyond current limitations?
    \item Can the best known weighted estimates for dyadic operators be improved?
    \item Why does the martingale BMO condition fail to characterize the boundedness of commutators, and is there a viable substitute that does?
\end{enumerate} 
The aim of this paper is to shed light on these questions.
These issues are not mere technicalities: nonhomogeneous measures naturally emerge in probability theory (via random measures), in geometric measure theory (through rectifiable measures), and in applied harmonic analysis (in the context of non-uniform sampling). A thorough understanding of operator bounds in such settings is fundamental to extending harmonic analysis beyond its traditional framework, with far-reaching applications to partial differential equations and signal processing, and deep connections to Hankel operators, weak factorization, and div–curl lemmas \cite{Wick2020}.

\subsection*{About Nonhomogeneous Settings}

In the classical doubling setting, the theory is remarkably clean. Commutators with Calder\'on-Zygmund operators are bounded if and only if the symbol belongs to BMO. Paraproducts with BMO symbols satisfy $L^p$ bounds for all $1 < p < \infty$. The powerful machinery of dyadic harmonic analysis, including the $T(1)$ theorem \cite{DavidJourne84} and paraproduct decompositions \cite{LPPW2010, HLW2016, HPW2018}, reduces continuous problems to dyadic ones, where control often follows from variants of the Carleson embedding theorem \cite{TolsaT1,NTV2003,HPTV}. Moreover, continuous BMO spaces can be recovered from dyadic ones through finite intersections or related constructions \cite{ GarnettJones, Mei}. 

This elegant theory collapses in the nonhomogeneous setting. Treil's impactful work \cite{TREIL2010} revealed that $L^p$ bounds for paraproducts depend essentially on $p$ through a ``$p$-Carleson packing'' condition, a phenomenon absent in the doubling case. Even more surprisingly, these bounds do not guarantee $L^p$ bounds for commutators with martingale transforms, which coincide with Haar multipliers in simpler settings. The endpoint case exhibits further pathologies: while Bonami et al. \cite{BonamiJXYZ} proved $H_1^b \to L^1$ estimates for commutators with martingale transforms, the analogous result for the dyadic Hilbert transform $\mathbb{S}$, introduced by Petermichl \cite{Petermichl}, requires an additional \emph{balanced condition} on the measure, introduced by Lopez-Sanchez, Martell, and Parcet \cite{LSMP}. Moreover, recovering continuous BMO spaces from dyadic ones only partially works for a specific class of BMO symbols \cite{TolsaRBMO,Conde-AlonsoParcet, CondeRBMO}, and this recovery depends essentially on a \emph{polynomial growth} condition on the underlying measure, which is entirely different from the balanced condition.
For those familiar with probability theory and the martingale setting, an intuitive justification of the ``Paradise Lost'' is that even the unit interval, endowed with the dyadic filtration and a non doubling measure, is \emph{not} a regular probability space, loosely meaning that measures of neighboring intervals do not necessarily relate well to each other. Whenever a dyadic operator reflects the interaction of dyadic cubes at different scales, there is no way to relate averages on the smaller cube to averages on the bigger cube.

\subsection*{Hints from Sparse Domination }

Sparse domination has emerged as the key tool for proving sharp weighted inequalities in modern harmonic analysis. The principle is elegant: if an operator can be dominated pointwise by sparse averages, then weighted estimates follow immediately. However, achieving sparse domination in nonhomogeneous settings has proven to be surprisingly difficult. Conde Alonso, Pipher, and the third author \cite{CPW} showed that classical sparse domination for $\mathbb{S}$ strikingly fails in the non-doubling setting, even when natural dyadic regularity assumptions on the measure are imposed, the so-called ``balanced condition". The authors instead proved a modified sparse domination for dyadic shifts: the modification, involving averages on neighboring intervals, highlighted the fundamental obstacles in the nonhomogeneous setting and the limitations of current sparse domination techniques in the general setting. By the same token, weighted inequalities require a stronger condition on the weight than the usual Muckenhoupt $A_p$ condition. This class of weights will be called the \emph{balanced $A_p$ class}. To further justify the relevance of sparse domination techniques, we also notice that a powerful version of ``continuous" sparse domination in the probabilistic setting was recently proved in \cite{contsparse} to obtain dimensionless $L^p$ bounds for the Bakry–Riesz vector on manifolds with bounded geometry.

In the specific case of paraproducts and commutators, recent work has developed sparse domination in wide-ranging homogeneous settings, including the Bloom weighted BMO setting \cite{HF2023} and commutators with continuous Calder\'{o}n-Zygmund operators \cite{LORR2017}. The non-homogeneous setting, by contrast, has remained largely unexplored. A key barrier has been Lacey's requirement \cite{Lacey_2017} of a packing condition on the symbol of dyadic paraproducts to obtain sparse domination, which is genuinely stronger than martingale BMO in nonhomogeneous settings. We emphasize this distinction: there exist specific non-doubling measures for which Lacey's packing condition is strictly stronger than the natural martingale BMO condition, and we provide an explicit example in Section \ref{Section 2}, while in the doubling case they always coincide. Surpassing this barrier to achieve sparse domination with only the BMO assumption has been an open problem, as existing techniques fundamentally relied on the extra structure provided by the packing condition.

\subsection*{The Dyadic Hilbert Transform}

Perhaps the most mysterious operator in this story is Petermichl's dyadic Hilbert transform $\Hilb$, defined by $\Hilb(h_I) = \text{sign}(I) h_{I^s}$ where $I^s$ is the dyadic sibling of $I$. Unlike the classical shift operator $\mathbb{S}$, this operator satisfies $\Hilb^2 = -I$ in perfect analogy with the classical Hilbert transform, making it natural for studying dyadic BMO in multiparameter and Banach-valued settings \cite{DKPS2023, DP2023}.

Yet $\Hilb$ exhibits baffling behavior in the nonhomogeneous setting. Recent work \cite{BCAPW} showed that even when $\mu$ is sibling balanced - a condition that \emph{characterizes} the boundedness of $\Hilb$ on $L^p(\mu)$ - the martingale BMO norm cannot be characterized by $\|[\Hilb,b]\|_{L^2(\mu) \to L^2(\mu)}$. They proved only a partial characterization:
\begin{equation}
\|b\|_{\mathcal{C}} \lesssim \|[\mathcal{H},b]\|_{L^2(\mu) \to L^2(\mu)} \lesssim \|b\|_{\BMO}, \label{eq:partialcharacterize}
\end{equation}
where $\|b\|_{\mathcal{C}}$ is the Carleson packing norm. The complete characterization of symbols yielding bounded commutators remained out of reach. Moreover, weighted estimates required introducing another subclass of weights denoted as $\widehat{A}_p$, and relied on the Cauchy integral trick, yielding:
\begin{equation}
\|[\Hilb,b]\|_{L^p(w) \to L^p(w)} \leq C(p,[w]_{\widehat{A}_p}) \|b\|_{\BMO}, \quad w \in \widehat{A}_p. \label{eq:hatweighted}
\end{equation}
It was left open whether the $\widehat{A}_p$ condition is sharp, while the operator $\Hilb$ itself was proved to obey weighted estimates for a strictly larger weight class.

\subsection{Main Contributions}

This paper provides answers to all the questions posed in the introduction and further explains some of these phenomena through two main results. 

 First, we prove sparse domination for dyadic paraproducts under only the natural BMO assumption, removing Lacey's packing condition entirely.

\begin{atheorem}[Sparse domination with BMO symbols] \label{thm:A} 
Let $\mu$ be an atomless Radon measure in $\R^n$ with $0< \mu(Q)< \infty$ for every $Q \in \D$, and $b \in \BMO$. Then any $T \in \{\Pi_b, \Pi^\ast_b, \Delta_b\}$ satisfies the following: for every $f \in L^1(\mu)$ compactly supported on $Q_0 \in \D$, there exists a dyadic sparse family $\mathcal{S}= \mathcal{S}(f)$ such that 
$$|Tf(x)| \lesssim \|b\|_\BMO \mathcal{A}_\mathcal{S}|f|(x), \quad \text{a.e. }x \in Q_0,$$
where the implicit constant depends on $T$ and $n$.
Consequently, for $T \in \{ \Pi_b, \Pi_b^\ast, \Delta_b \}$, every $1<p<\infty$ and $w \in A_p^\D(\mu)$, there exists a constant $C=C(p,n,T)$ such that
\[\|T\|_{L^p(w) \to L^p(w)} \leq C \|b\|_\BMO [w]_{A_p^\D(\mu)}^{\max \big(1, \frac{1}{p-1}\big)}.\]
\end{atheorem}

This immediately unlocks previously inaccessible weighted estimates for commutators.

\begin{acorollary}[Sharp weighted inequalities for Haar shifts]
\label{cor:B}
Suppose that $\mu$ is atomless and $\mathscr{H}$ is a generalized Haar system such that $(\mu,\mathscr{H})$ is balanced as in \cref{def of balanced pairs}. Let $1<p<\infty$, $b \in \BMO$, $w \in A_p^{b}(\mu)$ and $T$ a Haar shift of complexity $(s,t)$ with $s+t=N$. Then there exists a positive constant $C=C(p,N, \mu, \mathscr{H},T)$ depending exponentially on $N$ such that for all $f \in L^p(w)$:
\begin{equation*} 
\|[T,b]f\|_{L^p(w)} \leq C[w]_{A^\D_p(\mu)}^{\big(1 + \frac{1}{p-1} - \frac{2}{p}+\max \big(1, \frac{1}{p-1} \big)\big)}[w]_{A_p^b(\mu)}^{\frac{2^{N-1}}{p}}\|b\|_{\BMO} \|f\|_{L^p(w)}. 
\end{equation*}  
Moreover, if $T$ is $L^1$ normalized as in \cref{D: L^1 normalized}, we have 
\begin{equation*} 
\|[T,b]f\|_{L^p(w)} \leq C[w]_{A_p^\D(\mu)}^{2\max \big(1, \frac{1}{p-1}\big)}\|b\|_{\BMO} \|f\|_{L^p(w)},
\end{equation*}
where $C=C(p,N,T)$ depends linearly on the complexity.
\end{acorollary}

For the dyadic Hilbert transform specifically, we obtain even more refined estimates, that were previously inaccessible due to the lack of reverse Hölder inequalities for $A_p^{sib}$ weights. Our approach circumvents this obstacle entirely.

\begin{acorollary} \label{cor:C}
Suppose $\mu$ is sibling balanced and atomless. Let $1<p<\infty$, $b \in \BMO$, and $w \in A_p^{sib}(\mu)$. Then there exists a positive constant $C=C(p, \Hilb, \mu)$ such that for all $f \in L^p(w)$:
\begin{equation*} 
\|[\Hilb,b]f\|_{L^p(w)} \leq C[w]_{A^\D_p(\mu)}^{\big(1 + \frac{1}{p-1} - \frac{2}{p}+\max \big(1, \frac{1}{p-1} \big)\big)}[w]_{A_p^{sib}(\mu)}^{1/p}\|b\|_{\BMO} \|f\|_{L^p(w)}. 
\end{equation*}
\end{acorollary}

Our second main result is a complete characterization of the symbols $b$ for which the commutator $[\Hilb, b]$ is bounded on $L^p(\mu)$, revealing an unexpected phenomenon.

\begin{atheorem}[Characterization of Dyadic Hilbert Transform Commutator Bounds]
\label{thm:D}
Let $b$ be locally integrable, $1 < p < \infty$, and $\mu$ sibling balanced. The commutator $[\Hilb, b]$ extends to a bounded operator on $L^p(\mu)$ if and only if:
\begin{enumerate}
    \item %\label{Nec 1} 
    The symbol $b \in \bmo_{\alpha(p)}(\mu)$, where $\alpha(p)=\max(p,p')$;
    \item %\label{Nec 2} 
    The sequence $\beta=\{\beta_Q\}_{Q \in \D}$ with $\beta_Q= c_Q-c_{Q^s}$ and $c_Q= \langle b, h_Q^2 \rangle$ satisfies $\|\beta\|_{\ell^{\infty}}<\infty$.
\end{enumerate}
In other words, for $1< p < \infty$ and $\alpha(p):=\max(p,p')$:
\[ 
[\BMO]_p(\mu)=\{b \in \bmo_{\alpha(p)}(\mu): \beta \in \ell^\infty \}
\] 
and moreover 
\[ 
\BMO(\mu) \subsetneq [\BMO]_p(\mu) \subsetneq \bmo_p(\mu).
\] 
\end{atheorem}

This characterization is conceptually surprising: unlike the classical case where BMO characterizes commutator bounds uniformly in $p$, the nonhomogeneous setting exhibits a genuinely $p$-dependent hierarchy of symbol spaces. This suggests that nonhomogeneous harmonic analysis requires fundamentally new principles beyond classical intuition.

\begin{acorollary} \label{Main commutator symbols} 
Let \[B(\mu):= \{ b \in L_{\text{loc}}^2(\mu): \beta(b)= (\beta_Q(b))_Q \in \ell^\infty \};\]
\[[\BMO]_\infty(\mu):= \{b \in [\BMO]_2(\mu): \|[\Hilb,b]\|_{L^p(\mu) \to L^p(\mu)} < \infty \text{ for every } 1<p<\infty \}.\]
Then $\BMO(\mu) \subsetneq [\BMO]_\infty(\mu)$ and 
\[ 
[\BMO]_\infty(\mu)= B(\mu) \cap \bigcap_{p \geq 2} \bmo_p(\mu). 
\]  
\end{acorollary}

While these results address several questions in the nonhomogeneous setting, many related problems remain open. We will outline some of these at the end of the paper. 

\subsection*{Paper Organization}
The paper is organized as follows. In \cref{Section 2} we establish sparse domination for paraproducts and related operators, proving \cref{thm:A}. \cref{Section 2} also includes an explicit example where Lacey's packing condition is strictly stronger than martingale BMO. Section \ref{NewWeighted} recalls the correct framework to analyze Haar shifts and commutators in nonhomogeneous settings building on \cite{BBDPW, BCAPW}, and establishes \cref{cor:B}. The final section, \cref{Section 4}, provides the complete characterization of commutator symbols for the dyadic Hilbert transform, proving \cref{thm:D} and \cref{Main commutator symbols}.

\subsection*{Acknowledgments}
We would like to thank Jill Pipher and Jos\'e M. Conde Alonso for helpful discussions related to this work.

\section{Paraproducts and sparse domination } \label{Section 2}

Let \(\D\) be a dyadic grid in $\R^n$. In what follows, $\mu$ is a Borel measure on $\R^n$, $n\geq 1$, such that $0<\mu(Q)<\infty$ for every $Q\in \mathcal{D}$. We further assume that each quadrant has infinite measure. For any cube $Q \in \D$,  the dyadic expectation operator $\E_Q$ for a locally integrable function $f$ is
\[
\E_Q f(x) := \langle f \rangle_Q \1_Q(x)
\]
where $\langle f \rangle_Q = \frac{1}{\mu(Q)} \int_Q f(y) \, d\mu(y)$,  and the martingale difference operator $\Delta_Q$ is
\[
\Delta_Q f(x) :=\sum_{R \in \mathrm{ch}(Q)} \E_R f(x) - \E_Q f(x)=\sum_{R \in \mathrm{ch}(Q)} (\langle f \rangle_R - \langle f \rangle_Q) \1_R(x),
\]
where $\mathrm{ch}(Q)$ is the set of the $2^n$ dyadic children of $Q$. In what follows, given $Q \in \D$ we denote as $\widehat{Q}$ the dyadic parent of $Q$, i.e. the smallest cube in $\D$ that strictly contains $Q$.
 \begin{definition}
 Let $1 \leq p  < \infty$.
We say $b \in \mathrm{BMO_p}(\mu)$ if 
\begin{equation}\label{D: BMO}
\|b\|_{\BMO_p}:= \sup_{Q\in\D} \bigg( \frac{1}{\mu(Q)} \int_Q|b - \LL b \RR_{\widehat{Q}}|^p d\mu \bigg)^{\frac{1}{p}}< \infty.   
\end{equation}
\end{definition}
\begin{definition} Let $1 \leq p  < \infty$.
We say $b \in \mathrm{bmo}_p(\mu)$ if 
\begin{equation}\label{D: bmo}
\|b\|_{\bmo_p}:= \sup_{Q\in\D} \bigg( \frac{1}{\mu(Q)} \int_Q|b - \LL b \RR_{Q}|^p d\mu \bigg)^{\frac{1}{p}} < \infty.   
\end{equation}
\end{definition} 
Denote $\D(Q)=\{R \in \D: R \subseteq Q\}.$ As
\[(b-\ang{b}_Q)\1_Q(x)= \sum_{R \in \D(Q)} \Delta_Rb(x),\]
using orthogonality of martingale differences one can easily show that
$\|b\|_{\bmo_2}= \|b\|_{\mathcal{C}}$, where the latter is the Carleson norm
\[ \|b\|_{\mathcal{C}}=\sup_{Q \in \D} \bigg(\frac{1}{\mu(Q)} \sum_{R \in \D(Q)}\| \Delta_Rb\|_{L^2(\mu)}^2 \bigg)^{\frac{1}{2}}. \]
In general, if the measure is not dyadically doubling, we have $\BMO_p(\mu) \subsetneq \bmo_p(\mu)$, and these spaces coincide in the doubling setting. \par

Before introducing paraproducts, we record some known facts about $\BMO$ spaces in the martingale setting. The first is the celebrated John-Nirenberg inequality,  while the second is a direct characterization of $\BMO_p(\mu)$ for $1<p<\infty$. 
\begin{proposition}[John-Nirenberg inequality]
Suppose $b \in \BMO_p$ for some $1 \leq p<\infty$. Then $b \in \BMO_q$ for all $1 \leq q < \infty$, and moreover,
\begin{equation}
\|b\|_{\BMO_p} \sim_{p,q} \|b\|_{\BMO_q}
\label{eq:pequivalence}
\end{equation}
\end{proposition}
\begin{proposition}[\cite{TREIL2010}]
    For any $1 \leq p < \infty$ we have that $b \in \BMO_p$ if and only if the following properties hold: \begin{equation}\label{BMO prop 1} \int_{Q} \bigg( \sum_{R \in \D (Q)} | \Delta_Rb(x)|^2 \bigg)^{\frac{p}{2}} d\mu(x) \leq C \mu(Q), \quad \forall Q \in \D \end{equation} 
    \begin{equation} \label{BMO prop 2}
     \sup_{Q \in \D} \| \Delta_Qb \|_\infty < \infty.   
    \end{equation}
\end{proposition}
Note that \eqref{BMO prop 2} follows from \eqref{BMO prop 1} in the doubling setting, while this is not true in the general setting. When $p=2$ \eqref{BMO prop 1} is the usual Carleson packing condition $$\sum_{R \in \D(Q)} \| \Delta_Rb\|^2_{L^2(\mu)} \leq C \mu(Q), \quad \forall Q \in \D.$$ 
Since $\BMO_p=\BMO_1$ for every $1 \leq p <\infty$, we see that \begin{gather} \label{BMO norm equiv} \|b\|_{\BMO} \sim \|b\|_{\mathcal{C}} + \sup_{Q \in \D} \| \Delta_Qb\|_\infty.\end{gather}
Now we are ready to introduce paraproduct forms. 
\begin{definition}
    Let $b,f \in L^1_{\mathrm{loc}}(\mu)$.
A dyadic paraproduct associated to a symbol $b$ is defined as 
$$\Pi_bf(x)= \sum_{Q \in \D} \E_Qf(x) \Delta_Qb(x). $$
The adjoint paraproduct is defined as
$$\Pi_b^\ast f(x)= \sum_{Q \in \D} \E_Q(b \Delta_Qf)(x)= \sum_{Q \in \D} \E_Q\big( \Delta_Qb\Delta_Qf)(x). $$
Define also the following operators
\begin{align*}
\Delta_b f(x) &= \sum_{Q \in \D} \Delta_Q b(x) \Delta_Q f(x),\\
\Lambda_b^0 f(x) & = \Pi_f b(x)= \sum_{Q \in \D} \E_Qb(x) \Delta_Q f(x),\\
\Lambda_b(f)(x) & = \sum_{Q \in \D} \Delta_Q(b \Delta_Qf)(x).
\end{align*}
\end{definition}
Finally, we have the paraproduct decompositions, see \cite{TREIL2010},
\[b(x)f(x)= \Pi_bf(x)+ \Pi_b^\ast f(x) + \Lambda_bf(x)= \Pi_bf(x)+ \Delta_b f(x) + \Lambda^0_bf(x).\]
These two decompositions coincide in the Lebesgue measure case, but are genuinely different in the nonhomogeneous case. In the same paper, the continuity on $L^p$ of paraproduct forms has been studied extensively for $1<p<\infty$. In particular, the necessary and sufficient conditions for the boundedness of $\Pi_b$ essentially depends on $p$.
\begin{theorem}[\cite{TREIL2010}] \label{L^p bound of paraprod}
A paraproduct $\Pi_b$ is bounded on $L^p$ for $1<p<\infty$ if and only if it is bounded on characteristic functions, i.e. if and only if the following holds:
\begin{equation} \label{testing for Pi_b}\sup_{Q \in \D} \frac{1}{\mu(Q)}\int_Q \bigg| \sum_{R \in \D(Q)} \Delta_Rb(x) \bigg|^pd \mu(x) < \infty.\end{equation}
Moreover $\Delta_b$ is bounded on $L^p$ for $1<p<\infty$ if and only if $b \in \BMO(\mu).$
\end{theorem}
Note that condition \eqref{testing for Pi_b} coincides with $b \in \bmo_p(\mu)$, so we can rephrase it as 
\[ \Pi_b:L^p(\mu) \to L^p(\mu) \iff b \in \bmo_p(\mu).\]
The lack of John-Nirenberg inequality for $\bmo_p$ spaces explains why this condition depends on $p$.
Next, we recall some basic facts about sparse families. 
\begin{definition}

    Let $\mc{S} \subset \mc{D}$ be a family of dyadic cubes.
    \begin{enumerate}
        \item Let $0 < \eta < 1$. We say that $\mc{S}$ is $\eta$-sparse if for each $Q \in \mc{S}$, there exists some Borel set $E_Q \subset Q$ so that $\mu(E_Q) \geq \eta \, \mu(Q)$ and the collection $\{E_Q\}_{Q \in \mc{S}}$ is pairwise disjoint. 
        \item Let $\Lambda > 0$. We say that $\mc{S}$ is $\Lambda$-Carleson if for every sub-collection $\mc{S}' \subset \mc{S}$, we have 
        \[\sum_{Q \in \mc{S}'}\mu(Q) \leq \Lambda \,\mu\left(\bigcup_{Q \in \mc{S}'}Q\right).\]
    \end{enumerate}
\end{definition}

It was shown in \cite{Hanninen} that if the measure $\mu$ has no point masses then $\mc{S}$ is $\eta$-sparse if and only if $\mc{S}$ is $\eta^{-1}$-Carleson. See also \cite{lernernazarov}, \cite{Rey} and \cite{HonLor} for other proofs. \par Given a sparse family, a sparse operator is the positive operator defined as 
\[\mc{A}_\mc{S}f(x):= \sum_{Q \in \mc{S}} \E_Qf(x).\] 
The goal of this section is to prove the following. \begin{theorem}[Sparse domination for paraproducts and related operators] \label{thm:SparseDom} Let $\mu$ be an atomless Radon measure in $\R^n$ such that $0< \mu(Q)< \infty$ for every $Q \in \D$, and $b \in BMO$. Then any $T \in \{\Pi_b, \Pi^\ast_b, \Delta_b\}$ satisfies the following: for every $f \in L^1(\mu)$ compactly supported on $Q_0 \in \D$, there exists a dyadic sparse family $\mathcal{S}= \mathcal{S}(f)$ such that $$|Tf(x)| \lesssim \|b\|_\BMO \mathcal{A}_\mathcal{S}|f|(x), \quad \text{a.e. }x \in Q_0.$$
where the implicit constant depends on $T$, $n$.
\end{theorem} 
\begin{comment}\textcolor{purple}{NW: Small technicality that's always confused me: if we only assume $f \in L^1$ and compactly supported, how do we know the operator is well-defined and converges appropriately? One can assume that only finitely many terms are nonzero in the Haar expansion, but I don't immediately see how to pass to that from the general case.}
\textcolor{orange}{FD: you can always localize the operator to $Q_0$ for the weak $(1,1)$ argument at the beginning; for this (and to make sense of the other stopping conditions) a compactly supported $L^1$ function is enough. Also, a priori sparse shouldn't tell anything about pointwise finiteness, I think: you are just claiming a pointwise bound. Am I tripping?} \textcolor{purple}{NW: I guess I am still a little confused how we interpret the infinite sum in the definition of the paraproduct/Haar shift/other operator.}
\end{comment}
Before giving the proof, we provide some motivation. In the homogeneous case, pointwise sparse domination for paraproducts with symbol $b \in \BMO$ was proved in \cite{NPTV} and a similar proof appeared in \cite{Lacey_2017} in the non-homogeneous setting as long as the symbol $b$ satisfies the following packing condition:
\begin{equation} \label{E: lacey packing} \sum_{Q \in \D(Q_0)} \| \Delta_Qb\|^2_{\infty} \mu(Q) < \mu(Q_0); \quad \forall Q_0 \in \D.\end{equation}
While Carleson norm and \eqref{E: lacey packing} are equivalent if $\mu$ is doubling, and both conditions coincide with requiring $b \in BMO$, the second is stronger than the first if the measure is not doubling, as
\begin{gather*}
    \mu(Q) \|\Delta_Qb\|_\infty^2=\mu(Q) \max_{R \in \mathrm{ch}(Q)}| \langle f \rangle_R- \langle f \rangle_Q|^2 \geq \sum_{R\in \mathrm{ch}(Q)} | \langle f \rangle_R- \langle f \rangle_Q|^2 \mu(R)=\| \Delta_Q b\|^2_{L^2(\mu)}.
\end{gather*}
Moreover, we also see using \eqref{BMO norm equiv} that \eqref{E: lacey packing} is in general stronger than the condition $b \in \BMO.$ 
In particular, we give an explicit example of a measure $\mu$ and a symbol $b \in \BMO$ that does not satisfy \eqref{E: lacey packing}. We use an example of a non-doubling Borel measure $\mu$ via a dyadic construction originally due to \cite{LSMP}; see also \cite[Proposition 2.1]{CPW}. For $k \in \N$, let $I_k=[0, 2^{-k})$ and $I_k^b=[2^{-k}, 2^{-k+1})$ denote its dyadic sibling.  Let $\mu$ be uniform with density $1$ (i.e. the Lebesgue density) on $[0,1)^c$, while on the unit interval $[0,1)$ we define $\mu$ inductively with constant density on $I_k^b, k \geq 1$ according to the rules
$$ \mu(I_1)=\mu(I_1^b)=\frac{1}{2};$$
$$ \mu(I_k)= \left(\frac{k-1}{k}\right) \mu(I_{k-1}), \quad \mu(I_k^b)= \frac{1}{k} \mu(I_{k-1}), \quad k \geq 2.$$
Straightforward computations give
$$\mu(I_k) \sim \frac{1}{k}; \quad  \mu(I_k^b) \sim \frac{1}{k^2}; \quad \|h_{I_k}\|_{\infty} \sim k, \quad k \geq 1.$$
It is also easy to check that $\mu$ is atomless. 

 \begin{proposition} \label{prop: BMOStrictInLaceyPack}
Let $\mu$ be the Borel measure constructed above, and define
 $$b(x)= \sum_{k=1}^{\infty} \alpha_k h_{I_k}(x), \quad \alpha_k:= k^{-1/2} \mu(I_k)^{1/2}.$$ Then $b \in \BMO$, but 
 $$\sum_{k=1}^{\infty} \| \Delta_{I_k}b\|^2_{\infty} \mu(I_k)=+\infty.$$
 \begin{proof}
 We first show $\|b\|_{\mathcal{C}}< \infty.$ It suffices to verify the Carleson packing condition for intervals of the form $I_k$ only. Note that $\Delta_{I_k}b= \alpha_{I_k}h_{I_k}$, so $\|\Delta_{I_k}b\|_{L^2(\mu)}^2=\alpha_k^2=  \frac{\mu(I_k)}{k} \sim k^{-2}$, and $\Delta_J b=0$ if $J \neq I_k$ for some $k$. Fix a positive integer $k_0$, and observe
 \begin{align*} \sum_{I \subseteq I_{k_0}} \|\Delta_I b\|_{L^2(\mu)}^2 & = \sum_{k=k_0}^{\infty} k^{-1}\mu(I_k) \\
 &  \sim \sum_{k=k_0}^{\infty} k^{-2} \\
 & \sim \frac{1}{k_0} \sim \mu(I_{k_0}).
 \end{align*}
 On the other hand, for $k \in \Z_{+}$, $\|\Delta_{I_k} b \|_{\infty} \sim k \, \alpha_k \sim 1.$ This establishes $b \in \BMO$, but also $$\sum_{k=1}^{\infty} \| \Delta_{I_k}b\|^2_{\infty} \, \mu(I_k) \gtrsim \sum_{k=1}^{\infty} \frac{1}{k}=+\infty.$$
 \end{proof}
 \end{proposition}

We now show that the assumption on the symbol $b$ for sparse domination of the paraproduct $\Pi_b$ can in fact be relaxed to $b \in \BMO$.

\begin{lemma} \label{lemma for sparse}
For any $f \in L^1_{\mathrm{loc}}(\mu)$ and any dyadic cube $Q \in \D$, the following bound holds:
\[
\|\Delta_Q f\|_{L^1(\mu)} \le 2 \int_Q |f(x)| \, d\mu(x).
\]
Moreover, we have
\[
\left| \E_{Q}\big( \Delta_{Q}b\Delta_{Q}f \big)(x)\right| \leq 2 \|b\|_\BMO\langle |f| \rangle_{Q}.
\]
\end{lemma}

\begin{proof}
As the children of $Q$ are disjoint we have
\begin{align*}
    \|\Delta_Q f\|_{L^1(\mu)} &= \int_Q \left| \sum_{R \in \mathrm{ch}(Q)} (\langle f \rangle_R - \langle f \rangle_Q) \1_R(x) \right| d\mu(x) = \sum_{R \in \mathrm{ch}(Q)} \mu(R) |\langle f \rangle_R - \langle f \rangle_Q|.
\end{align*}
Using the triangle inequality we get:
\begin{align*}
    \mu(R) |\langle f \rangle_R - \langle f \rangle_Q| \le \mu(R) (|\langle f \rangle_R| + |\langle f \rangle_Q|) \leq \int_R |f| \, d\mu + \mu(R) \langle|f|\rangle_Q.
\end{align*}
Summing over $R \in \mathrm{ch}(Q)$ completes the first part. Also, by Hölder's inequality and $b \in \BMO$
\begin{align*}
    \left| \E_{Q}\big( \Delta_{Q}b  \Delta_{Q}f \big)(x) \right|= \frac{1}{\mu(Q)} \left| \int_{Q} \Delta_{Q}b(y) \Delta_{Q}f(y) \, d\mu(y) \right| \leq \frac{\|b\|_\BMO\|\Delta_{Q} f\|_{L^1(\mu)}}{\mu(Q)} \leq 2 \|b\|_\BMO \langle |f| \rangle_{Q}.
    \end{align*}
\end{proof}
We now introduce the nonhomogeneous Calderón-Zygmund decomposition. 
\begin{lemma}\cite{CPW} \label{L: CZD}
Let $f:\R^n\rightarrow \R$ with $f \in L^1(\mu)$ supported in $Q_0 \in \D$. Then, for every $\lambda>0$ there exist functions $g,b$ such that $f=g+b$ and the following holds
\begin{enumerate}
    \item There exists a family of pairwise disjoint intervals $\{Q_k\}_k \subset \D(Q_0)$ such that \[b= \sum_{k \in \N} b_k; \quad \quad b_k= f\1_{Q_k}- \LL f\1_{Q_k}\RR_{\widehat{Q_k}}\1_{\widehat{Q_k}}.\]
    In particular, for every $k$, $\|b_k\|_{L^1(\mu)} \lesssim \int_{Q_k} |f| d\mu$ and $b_k$ has zero mean on $\widehat{Q_k}$.  
    \item We have that $g \in L^p(\mu)$ for every $1 \leq p < \infty $ and $\|g\|^p_{L^p(\mu)} \lesssim_p \lambda^{p-1}\|f\|_{L^1(\mu)}$. Moreover, $g \in \BMO(\mu)$ and $\|g\|_{\BMO} \leq \lambda.$
\end{enumerate}
\end{lemma}
\begin{definition}
    Let $T=\sum_{Q \in \D}T_Q$ be a dyadic operator. The maximal truncation of $T$ is 
$$T^\#f(x):= \sup_{Q_0 \ni x} \bigg| \sum_{Q_0 \subsetneq Q} T_Qf(x) \bigg|,$$
where the supremum is taken over  $Q_0 \in \D$.
\end{definition}
To prove sparse domination, we need to control maximal truncations of paraproducts. 
\begin{proposition} \label{weak (1,1) max. trunc}
    Let $b \in \BMO$ and $T \in \{ \Pi_b, \Pi_b^*, \Delta_b\}.$ Then for every $1<p<\infty$
    \[ \|T^\# \|_{L^p(\mu) \to L^p(\mu)} \lesssim   \|b\|_\BMO.\] 
   and  \[\|T^\#\|_{L^1(\mu) \to L^{1,\infty}(\mu)} \lesssim \|b\|_\BMO.\]
\end{proposition}

\begin{proof}

    The following dyadic Cotlar's type inequality was shown in \cite{HF2023}: \[\Pi_b^\#f(x) \leq M_\D(\Pi_bf)(x), \qquad \forall \, x \in \R^n\]
    where $M_\D$ is the dyadic maximal function, and $L^p$ boundedness follows. Recall that $\Delta_bf$ is $L^p(\mu)$ bounded if and only if $b \in \BMO$ and $\|\Delta_b\|_{L^p(\mu) \to L^p(\mu)} \sim \|b\|_\BMO.$ Then
\begin{align*}\E_{Q_0}(\Delta_b f)(x)=& \sum_{Q_0 \subsetneq Q} \E_{Q_0}(\Delta_Qb\Delta_Qf)(x)+ \E_{Q_0}\bigg(\sum_{Q \in \D (Q_0)} (\Delta_Qb\Delta_Qf)(x)\bigg) \\ =& \sum_{Q_0 \subsetneq Q} \Delta_Qb(x)\Delta_Qf(x)+ \E_{Q_0}\bigg(\sum_{Q \in \D (Q_0)} (\Delta_Qb\Delta_Qf)(x)\bigg),\end{align*}
since the first sum is constant on $Q_0$. For $ x \in Q_0$ and $b \in \BMO$, \eqref{BMO prop 2} gives for $1<q<\infty$
\begin{align*}
\frac{1}{\mu(Q_0)}\int_{Q_0} \sum_{Q \in \D(Q_0)} \Delta_Q b \Delta_Q f & \leq \frac{1}{\mu(Q_0)}\int_{Q_0} \bigg(\sum_{Q \in \D(Q_0)}  |\Delta_Qb|^2 \bigg)^{\frac{1}{2}}  \bigg(\sum_{Q \in \D(Q_0)}  |\Delta_Qf|^2 \bigg)^{\frac{1}{2}} \, dx  \\ 
& \leq  \bigg( \frac{1}{\mu(Q_0)}\int_{Q_0} \big(\sum_{Q \in \D(Q_0)}  |\Delta_Qb|^2 \big)^{\frac{q'}{2}} \, dx\bigg)^{\frac{1}{q'}} (\ang{Sf^q}_{Q_0})^{\frac{1}{q}} \\
& \lesssim_q \|b\|_\BMO (\ang{Sf^q}_{Q_0})^{\frac{1}{q}}, 
\end{align*}
where $Sf$ is the dyadic square function. Therefore, for every $1<q<\infty$
\begin{equation}\label{dyadic cotlar}\Delta_b^\#f(x) \leq M_\D(\Delta_bf)(x)+ C_q\|b\|_\BMO M^q_\D(Sf)(x),\end{equation}
where $M_\D^qf(x)= \sup_{Q_0 \in \D} \langle |f|^q\rangle_{Q_0}^{\frac{1}{q}}\1_{Q_0}(x).$
Note that the first term is $L^p$ bounded for every $1<p<\infty$ and the second is $L^p$ bounded for $p>q$. Then for every $1<p<\infty$, choosing $1<q<p$ we conclude that \[\|\Delta_b^\#\|_{L^p(\mu) \to L^p(\mu)}\lesssim_p \|b\|_\BMO.\]
The argument for $(\Pi_b^*)^\#$ is essentially the same, since 
\begin{align*}\E_{Q_0}(\Pi_b^* f)(x)=& \sum_{Q_0 \subsetneq Q} \E_{Q_0}(\E_Q(\Delta_Qb\Delta_Qf))(x)+ \E_{Q_0}\bigg(\sum_{Q \in \D (Q_0)} \E_Q(\Delta_Qb\Delta_Qf)(x)\bigg) \\ =& \sum_{Q_0 \subsetneq Q}\E_Q(\Delta_Qb\Delta_Qf)(x)+ \E_{Q_0}\bigg(\sum_{Q \in \D (Q_0)} \E_Q(\Delta_Qb\Delta_Qf)(x)\bigg) \\ =& \sum_{Q_0 \subsetneq Q}\E_Q(\Delta_Qb\Delta_Qf)(x)+ \E_{Q_0}\bigg(\sum_{Q \in \D (Q_0)} (\Delta_Qb\Delta_Qf)(x)\bigg),\end{align*}
where in the last equality we used that 
\[\Delta_Q b\Delta_Qf= \Delta_Q(\Delta_Q b\Delta_Qf)+\E_Q(\Delta_Q b\Delta_Qf)\] and the fact that $\E_{Q_0}(\Delta_Q(\Delta_Qb\Delta_Qf))=0$ for $Q \in \D(Q_0).$ This leads to the same behaviour as in \eqref{dyadic cotlar} with $\Pi_b^*$ instead of $\Delta_b$ and to $L^p$ boundedness with operator norm depending on $\|b\|_\BMO$. \par

Now we turn to weak $(1,1)$ boundedness. Let $\lambda>0$, $f$ be compactly supported and $f=g+ \beta$ the Calder\'on-Zygmund decomposition given in Lemma \ref{L: CZD} of $f$ at height $\lambda.$ We deal with $\Pi_b^\#$ first: by the $L^2$ boundedness of maximal truncations
 \begin{align*}
     \mu \left(\left\{ x: |\Pi^\#_bf(x)| > \lambda \right\}\right) \leq &  \mu \left( \left\{ x: |\Pi^\#_bg(x)| > \lambda/2 \right\}\right)+ \mu \left( \left\{ x: |\Pi^\#_b \beta(x)| > \lambda/2 \right\}\right) \\ \leq & \frac{C}{\lambda} \|b\|_{\mathrm{BMO}(\mu)}\|f\|_{L^1(\mu)}  + \mu \left(\left\{ x: |\Pi^\#_b \beta(x)| > \lambda/2 \right\}\right),
 \end{align*}
 Hence we only have to estimate the second term. Observe that, if $\widehat{Q_j} \subseteq Q$, then 
 $$ \LL \beta_j \RR_{Q}= \LL f \1_{Q_j} \RR_Q - \LL f\1_{Q_j}\RR_{\widehat{Q_j}} \frac{\mu(\widehat{Q_j})}{\mu(Q)}=0,$$
 Therefore $$\Pi_b^\# \beta(x) \leq  \sup_{Q_0 \ni x} \bigg| \sum_j \sum_{Q_0 \subsetneq Q \subseteq Q_j}\E_Q \beta_j(x) \Delta_Qb(x)\bigg|.$$
In particular, $\Pi_b^\#(\beta)$ is supported in $\bigcup_jQ_j$, so 
$$\mu \left(\left\{ x: |\Pi^\#_b \beta(x)| > \lambda/2 \right\}\right)\leq \mu \left(\bigcup_j Q_j\right) \leq \frac{\|f\|_{L^1(\mu)}}{\lambda}.$$
This concludes that $\|\Pi_b^\#\|_{L^{1}(\mu) \to L^{1,\infty}(\mu)} \lesssim \|b\|_\BMO.$ \par
Similarly, for $\Delta_b^\#$ we only need to study $$\mu \left( \left\{ x: |\Delta^\#_b \beta(x)| > \lambda/2 \right\}\right).$$ Since \(\Delta_Q(\beta_j) \neq 0\) if and only if \( Q \subseteq \widehat{Q_j},\) we get $$\Delta_b^\# \beta(x) \leq  \sup_{Q_0 \ni x} \bigg| \sum_j \sum_{Q_0 \subsetneq Q \subseteq Q_j} \Delta_Qb(x) \Delta_Q \beta_j (x)\bigg|+\sum_j \big|\Delta_{\widehat{Q_j}}b(x) \Delta_{\widehat{Q_j}}\beta_j(x)\big|= A(x)+B(x).$$
As before, we have
$$\mu \left(\left\{x: A(x) > \frac{\lambda}{4} \right\}\right) \leq \mu \left(\bigcup_j Q_j\right) \leq \frac{\|f\|_{L^1(\mu)}}{\lambda}.$$
Using \cref{lemma for sparse} and $\ang{\beta_j}_{\widehat{Q_j}}=0$, combined with \cref{L: CZD} and $b \in \BMO$
\begin{align*}
    \|B\|_{L^1}\leq & \|b\|_\BMO \sum_j \|\Delta_{\widehat{Q_j}}\beta_j\|_{L^1(\mu)} \\\leq &\|b\|_\BMO \sum_j \int_{\widehat{Q_j}} |\beta_j| \\ \leq &\|b\|_\BMO \sum_j \|\beta_j\|_{L^1(\mu)} \leq \|b\|_\BMO \|f\|_{L^1(\mu)}.
\end{align*}
We finally get 
$$\mu \left(\left\{x: B(x) > \frac{\lambda}{4} \right\}\right) \lesssim \frac{\|b\|_\BMO\|f\|_{L^1(\mu)}}{\lambda}.$$
The same argument used for $\Delta_b^\#$ works for $(\Pi_b^*)^\#$ by noticing that
\[\|\E_Q(\Delta_Qb\Delta_Q \beta_j)\|_{L^1} \leq \|\Delta_Qb\Delta_Q\beta_j\|_{L^1} \leq \|b\|_\BMO \|\Delta_Q\beta_j\|_{L^1}.\]
\end{proof}
\begin{proof}[Proof of \cref{thm:SparseDom}] 
Let $T \in \{ \Pi_b, \Pi_b^*, \Delta_b \}$. We can assume $Q_0 \in \D$, otherwise we can replace $Q_0$ with a larger cube. Note that for a.e. $x\in Q_0$,
    \[Tf(x) =\sum_{Q \in \mc{D}(Q_0 )}T_Qf(x) + \sum_{Q \in \mc{D}: Q_0 \subsetneq Q}T_Qf(x)=:T^{Q_0}f(x) +\widetilde{T}f(x).\]
From Proposition \ref{weak (1,1) max. trunc}, for any $C>4 \|T^\#\|_{L^1(\mu) \to L^{1,\infty}(\mu)}$
    \begin{align*}
        \mu(\{x \in Q_0: |\sq{T}f(x)| > C\ang{|f|}_{Q_0}\}) \leq \frac{1}{4}\mu(Q_0).
    \end{align*}
On the other hand, for any such $T$ we have that $T_Qf(x)$ is constant on $Q_0$ when $Q_0 \subsetneq Q$, hence $\widetilde{T}f(x)$ is constant as well. Therefore, choosing $C$ as before we argue \[|\widetilde{T}f(x)| \leq C \ang{|f|}_{Q_0} \quad \text {on $Q_0$,} \] and it suffices to bound the local operator $T^{Q_0}$. \\
\begin{comment}Note that $\widetilde{T}f(x)\1_{Q_0}(x)=\langle \Pi_bf\rangle_{Q_0}\1_{Q_0}(x)$ and so it is constant on $Q_0$. \\\end{comment} 
For any $T$ as above, let $B(Q_0):=\{Q_j\}_j$ the set of maximal intervals in $\D(Q_0)$ such that
\begin{equation}\label{E: stop}
    \LL |f|\RR_{Q_j} > C_1 \LL |f| \RR_{Q_0} \quad \text{ or } \quad \bigg| \sum_{Q_j \subsetneq Q \subseteq Q_0} T_Q(f\1_{Q_0})(x) \bigg| > C_2 \LL|f| \RR_{Q_0} \quad \text{ on $Q_j$.}
\end{equation}
Denote $B^1(Q_0)$ the intervals in $B(Q_0)$ such that the first stopping condition holds, and $B^2(Q_0)$ the intervals in $B(Q_0)$ such that the second holds. Consider the operator  \[T^1 = \sum_{Q \in \mc{D}(Q_0) \setminus \bigcup_{Q_j \in B^2(Q_0)}\mc{D}(Q_j)}T_Q.\]
    Then if $x \in Q_j$ we have $|T^1(f\1_{Q_0})(x)| > C_2 \ang{|f|}_{Q_0}$ by \eqref{E: stop}. Choosing $C_2>4\|T^\#\|_{L^1(\mu) \to L^{1,\infty}(\mu)}$ 
    \[\sum_{Q_j \in B^2(Q_0)}\mu(Q_j) \leq \mu\left(\{x \in Q_0 : |T^1(f\1_{Q_0})(x)|>C_2 \ang{|f|}_{Q_0}\}\right) \leq \frac{1}{4}\mu(Q_0).\] 
    Similarly, we can use the weak $(1,1)$ bound for the dyadic Hardy-Littlewood maximal function to bound the sum of the measures of the cubes satisfying the first stopping condition in \eqref{E: stop} by $\frac{1}{4}\mu(Q_0)$. Altogether we get
    \[\sum_{Q_j \in \mc{B}(Q_0)}\mu(Q_j) \leq \frac{1}{2}\mu(Q_0).\]
We now form a sparse family $\mc{S}$ in the standard way: set $\mc{B}_0(Q_0) := \{Q_0\}$ and inductively define 
    \[\mc{B}_k(Q_0) := \bigcup_{Q \in \mc{B}_{k-1}(Q_0)}\mc{B}(Q).\]
    The family
    \[\mc{S} = \bigcup_{k=0}^{\infty} \mc{B}_k(Q_0)\]
    is then $\frac{1}{2}$-sparse. 
Finally
    $$\big|T^{Q_0}(f)(x)\1_{Q_0}(x)\big| \leq \big|T^{Q_0}(f)(x)\1_{Q_0 \setminus \bigcup_j Q_j}(x)\big| + \sum_{j} \big|T^{Q_0}(f)(x)\1_{Q_j}(x)\big|.$$

The first term is controlled by $C_2 \LL |f| \RR_{Q_0}$. Moreover, for $x \in Q_j$ \begin{equation} \label{E: sparse iteration}|T^{Q_0}(f)(x)|\leq  |T_{\widehat{Q_j}}f(x)|+ \bigg|\sum_{\widehat{Q_j} \subsetneq Q \subseteq Q_0} T_Qf(x)\bigg| + |T^{Q_j}(f\1_{Q_j})(x)\1_{Q_j}(x)|.\end{equation}
Then, by $\eqref{E: stop}$, the second term is controlled by $C_2 \LL |f| \RR_{Q_0}.$ Hence, to iterate the procedure, we only need to control the first term for any given $T \in \{\Pi_b, \Pi^\ast_b, \Delta_b\}.$ By Lemma \ref{lemma for sparse} and \eqref{E: stop}, since $\widehat{Q_j}$ was not selected, if $x \in Q_j$ we have for $C=C\big(\|M_\D\|_{L^1(\mu) \to L^{1,\infty}(\mu)},\|T^\#\|_{L^1(\mu) \to L^{1,\infty}(\mu)}\big)$
\begin{align*}
 |\LL f \RR_{\widehat{Q_j}}\Delta_{\widehat{Q_j}}b(x)| & \leq \|b\|_{\mathrm{BMO}(\mu)} \LL |f| \RR_{\widehat{Q_j}}\1_{Q_j}(x) \leq C \|b\|_{\mathrm{BMO}(\mu)} \LL |f| \RR_{Q_0}\1_{Q_j}(x);\\
 \left| \E_{\widehat{Q_j}}\big( \Delta_{\widehat{Q_j}}b\Delta_{\widehat{Q_j}}f \big)(x) \1_{Q_j}(x)\right| & \leq 2 \|b\|_\BMO\langle |f| \rangle_{\widehat{Q_j}}\1_{Q_j}(x) \leq C\|b\|_\BMO \langle |f| \rangle_{Q_0}\1_{Q_j}(x);\\
 |\Delta_{\widehat{Q_j}}b(x) \Delta_{\widehat{Q_j}} f(x)\1_{Q_j}(x)| & \leq \|b\|_{\BMO}\|\Delta_{\widehat{Q_j}} f(x)\1_{Q_j}(x)\|_\infty \leq \|b\|_\BMO \big(\langle |f|\rangle_{Q_j}+ C\ang{|f|}_{Q_0}\big).
 \end{align*}
We obtain for any $T$ as above and 
 \begin{equation} \label{iteration}
\big|T^{Q_0}(f)(x)\1_{Q_j}(x)\big| \leq C \|b\|_{\mathrm{BMO}(\mu)} \big(\LL |f| \RR_{Q_0} + \ang{|f|}_{Q_j}\big) + \sum_j |T^{Q_j}(f\1_{Q_j})(x)\1_{Q_j}(x)|     
 \end{equation} 
and we can iterate the procedure for $T^{Q_j}$, for any $Q_j \in B(Q_0)$. Notice that from \eqref{iteration} the average over any $Q \in \mc{S}$ will appear at most twice. We can conclude that for any $T \in \{\Pi_b, \Pi^\ast_b, \Delta_b\}$ and $f \in L^1(\mu)$ supported on $Q_0$, there exists a sparse family $\mathcal{S}=\mathcal{S}(T,f)$ such that $$|Tf(x)| \lesssim \|b\|_\BMO \mathcal{A}_\mathcal{S}|f|(x), \quad \text{a.e. }x \in Q_0.$$ \end{proof}

\begin{corollary}
    Let $1<p<\infty $ and $w \in A^\D_p(\mu)$, i.e. \[ [w]_{A^\D_p(\mu)}:= \sup_{Q \in \D} \ang{w}_Q \ang{\sigma}_Q^{p-1}<\infty, \]
where $\sigma=w^{1-p'}$ is the $p$-dual weight of $w$. For any $T \in \{ \Pi_b, \Pi_b^\ast, \Delta_b \}$,  $1<p<\infty$ and $ w \in A_p^\D(\mu)$ there exists a constant $C=C(p,n,T)$ \[\|T\|_{L^p(w) \to L^p(w)} \leq C \|b\|_\BMO [w]_{A_p^\D(\mu)}^{\max \big(1, \frac{1}{p-1}\big)}.\]
\end{corollary}

\begin{remark}
    The same strategy of \cref{thm:SparseDom} can be applied almost verbatim to vector valued paraproduct forms. If $T$ is a linear operator acting on scalar valued functions and $f: \R^n \to \R^d$, we abuse notation writing $Tf$ instead of $(T \otimes I_d)(f)$, where $$(T \otimes I_d)(f)=(Tf_1, \dots, Tf_d).$$ The convex body average $\dang{f}_Q$ is the compact, convex and symmetric set defined as the image of the unit ball of $L^\infty(Q)$ under the bounded linear functional defined by the pairing with $f$
\[\dang{f}_Q := \{ \LL f \psi \RR_Q, \ \psi: Q \to \R, \|\psi \|_\infty \leq 1 \},\]
 where 
 \[ \LL f \psi \RR_Q := \frac{1}{\mu(Q)} \int_Q f(x) \psi(x) d\mu(x),\] 
is the vector whose $i$-th component is $\langle f_i \psi \rangle_Q$, for $i=1, \dots, d.$

Then one can follow the same proof as in \cite[Theorem 3.13, pag. 18]{BBDPW} to prove for any $T \in \{\Delta_b, \Pi^\ast_b, \Pi_b\}$ that \[Tf(x) \in C\sum_{Q \in \mc{S}}\dang{f}_Q \1_{Q}(x) \quad \text{ on } Q_0. \]
As an application it follows that for any $1 < p < \infty$ and $W \in A_p$, we have 
    \[ \|T\|_{L^p(W) \to L^p(W)} \lesssim_{p,d} [W]_{A_p}^{1 + \frac{1}{p-1} - \frac{1}{p}}.\]   
We refer to \cite{BBDPW} for more details.
\end{remark}

\section{Weighted Inequalities for Commutators with dyadic shifts} \label{NewWeighted}
In this section, we see how Theorem \ref{thm:SparseDom} leads to the following strengthened weighted inequalities for the commutator $[T,b]$ with dyadic shifts. Indeed, this approach removes the key obstacle of requiring a reverse H\"{o}lder inequality for the weight $w.$
Recall that, for a fixed dyadic grid $\D$, we assume for simplicity that $\mu$ is a Radon measure on $\R^n$ such that $0<\mu(Q)< \infty$ for any $Q \in \D$. This is not a structural restriction and can be removed; see for example the discussion in \cite{LSMP}, \cite{TREIL2010} and \cite{BBDPW}. We further suppose that $\mu$ is atomless. Many of the following definitions are quoted verbatim from \cite{BBDPW}.
\subsection{Haar shifts: modified sparse domination and weighted inequalities}

\begin{definition}\label{gen haar systems}
We say $\mathscr{H}=\{h_Q\}_{Q \in \D}$ is a generalized Haar system in $\R^n$ if the following holds:
\begin{enumerate}
    \item for every $Q \in \D$ we have $\mathrm{supp}(h_Q) \subset Q$;
    \item for every $R \in \mathcal{D}(Q)$, $R\subsetneq Q$, $h_Q$ is constant on $R$; in particular \[h_Q(x)=\sum_{R \in \mathrm{ch}(Q)} \alpha_R \1_R(x);\] 
    \item for every $Q \in \D$, $h_Q$ has zero mean, i.e. $\int_Q h_Q(y)d\mu(y)=0$;
    \item for every $Q \in \D$, we have $\|h_Q\|_{L^2(\mu)}=1.$

\end{enumerate}
Furthermore, we say $\mathscr{H}$ is standard if 
\begin{equation} \label{D: standard}
    \Xi \left[\mathscr{H}, 0,0 \right]:= \sup_{Q \in \D} \|h_Q\|_{L^1(\mu)} \|h_Q\|_{L^\infty(\mu)} <\infty.   
   \end{equation}
\end{definition}

\begin{remark}
A generalized Haar system $\mathscr{H}$ is in general an orthonormal set in $L^2(\R^n)$, not necessarily an orthonormal basis for $L^2(\R^n)$. However, we still have
\begin{equation} \label{Bessel}\sum_Q |\langle f, h_Q\rangle|^2 \leq \|f\|_{L^2(\mu)}^2.\end{equation}
\end{remark}
\begin{definition}\label{def: vectorhaarshift}
A generalized Haar shift $T$ of complexity $(s,t)$ acting (a priori) on $f \in L^2(\R^n)$ takes the form
\begin{gather} \label{D: dyadic shifts}Tf(x)= \sum_{Q \in \D} T_Qf(x):= \sum_{Q \in \D} \int_Q K_Q(x,y)f(y)d\mu(y),
\end{gather}
where
\begin{gather*}
K_Q(x,y)=\sum_{\substack{J \in \D_s(Q)\\ K \in \D_t(Q)}} c_{J,K}^Qh_J(y)h_K(x), \quad \text{ and }\quad \sup_{Q,J,K} | c_{J,K}^Q| \leq 1.\end{gather*} 
If, in addition, one has $\inf_{Q,J,K} | c_{J,K}^Q|>0$, then we say that $T$ is a non-degenerate (vector) Haar shift of complexity $(s,t)$.
\end{definition} 
It is straightforward to check that \eqref{Bessel} implies that for every $(s,t) \in \N^2$ every generalized Haar shift of complexity $(s,t)$ is bounded on $L^2(\mu)$. \par
\begin{definition}\label{D: L^1 normalized}
We say a generalized Haar shift $Tf(x)= \sum_{Q \in \D}T_Q f(x)$ defined as in Definition \cref{D: dyadic shifts} is $L^1$ normalized if \begin{equation} \label{E: L1 norm}\|K_Q\|_{\infty}  \lesssim_{\mu} \frac{1}{\mu(Q)}, \quad \text{ for any $Q \in \D$}.\end{equation}
\end{definition}
This subclass of shifts was already studied in \cite{BBDPW}. In the doubling setting, the decay of the kernel, depending on $\frac{1}{\mu(Q)}$, easily follows from norm properties of Haar functions, and the implicit constant usually depends exponentially on the complexity if the shift has merely $\ell^\infty$ coefficients. However, the dyadic operators appearing in applications - say, in representation theorems -  have extra normalization which justifies \eqref{E: L1 norm}. In the nonhomogeneous setting, the kernel of a shift with merely $\ell^\infty$ coefficients does not even have the usual measure decay.

We now come to the balanced condition. Given a pair $(\mu, \mathscr{H})$, where $\mathscr{H}$ is a generalized Haar system and $\mu$ as above, define the quantities 
\begin{equation}\label{def:m(Q)}
    m(Q)= m_{\mu, \mathscr{H}}(Q) := \|h_Q\|_{L^1(\mu)}^2.
\end{equation}
\begin{definition} \label{def of balanced pairs} 
    We say that a pair $(\mu, \mathscr{H})$ is balanced if $\mathscr{H}$ is standard and
    \begin{equation} \label{balanced_higher_dim}
    m(Q) \sim m(\widehat{Q}), \quad \text{ for every $Q \in \D$}
    \end{equation}
\end{definition}
The following proposition was proved in \cite{BBDPW}.

\begin{proposition}
    If a pair $(\mu, \mathscr{H})$ is balanced, every generalized Haar shift defined with respect to $\mathscr{H}$ is weak $(1,1)$ and bounded on $L^p(\mu)$ for any $1<p<\infty.$
    If a generalized Haar shift defined with respect to a generalized Haar system $ \mathscr{H} $ and any measure $\mu$ is $L^1$ normalized, then it is weak $(1,1)$ and bounded on $L^p(\mu)$ for any $1<p<\infty.$
\end{proposition}
Note that given a Radon measure $\mu$ as before one can build two Haar systems $\mathscr{H}$ and $\widetilde{\mathscr{H}}$ such that $(\mu, \mathscr{H})$ is balanced but $(\mu, \widetilde{\mathscr{H}})$ is not, see \cite[Section 4.3]{LSMP}. On the other hand, it is easy to show that if $(\mu,\mathscr{H})$ is balanced, then  $$m(Q)\sim \min\{\mu(R)\colon R\in \mathrm{ch}(Q)\}.$$ This means that for two generalized Haar systems $\mathscr{H}$ and $\widetilde{\mathscr{H}}$ such that $(\mu, \mathscr{H})$ and $(\mu, \widetilde{\mathscr{H}})$ are balanced pairs, we have that 
\begin{equation} \label{equiv of balanced constants}m_{\mu, \mathscr{H}}(Q) \sim m_{\mu, \widetilde{\mathscr{H}}}(Q), \quad \text{ for every } Q \in \D.\end{equation}
For a deeper treatment of balanced pairs, see \cite{BBDPW}.
\begin{remark}
Let us comment on the generality of the previous definitions. Recall that  \[\Delta_Q:L^2(\mu) \to \Delta_QL^2(\mu)\] is an orthogonal projection on the $2^{n}-1$ dimensional vector space $\Delta_QL^2(\mu)$, and it holds that \[L^2(\mu)= \bigoplus_{Q \in \D} \Delta_QL^2(\mu).\]
In particular,  $\Delta_QL^2(\mu)$ is a linear span of the set $V_Q= \{h_Q^1,\dots, h_Q^{2^n-1} \},$ where each $h_Q^j$ verifies properties $(1)-(4)$ in Definition \ref{gen haar systems}, and consequently $L^2(\mu)$ is spanned by the Haar basis \[\mathscr{H}=\bigcup_{Q \in \D} V_Q.\] 

Consider any Haar shift of the form \begin{equation} \label{martingale haar shifts} T=\sum_Q T_Q, \qquad  T_Q = \sum_{\substack{J \in \D_s(Q)\\ K \in \D_t(Q)}}\Delta_J T_{J,K} \Delta_K, \end{equation} and $T_{J,K}: \Delta_KL^2(\mu) \to \Delta_JL^2(\mu)$ is uniformly bounded. Expanding the Haar basis we get 
\[\Delta_J T_{J,K} \Delta_K f = \sum_{j,k=1}^{2^n-1} \alpha^T_{j,k} \langle f, h_K^k\rangle h_J^j(x), \qquad \alpha_{j,k}^T:=\ang{T_{J,K}h_K^k,h_J^j} \in \ell^\infty .\]\par 
In other words \[Tf(x)= \sum_{j,k=1}^{2^n-1}T^{j,k}f(x), \qquad  T^{j,k} f(x):=\sum_{Q \in \D} \sum_{\substack{J \in \D_s(Q)\\ K \in \D_t(Q)}} \alpha^T_{j,k} \langle f, h_K^k\rangle h_J^j(x), \]
and it suffices to study $T^{j,k}$ for each $j,k= 1, \dots, 2^n-1.$
This way,  we can see any such Haar shift as a finite sum (depending only on the dimension) of generalized Haar shifts, each corresponding to the generalized Haar system obtained by properly choosing one single Haar function for every dyadic cube. Notice that to study more general martingale operators as in \eqref{martingale haar shifts} we therefore need to require that \[ m(Q) \sim \|h^j_Q\|_{L^1(\mu)} \sim \|h^i_{\widehat{Q}}\|_{L^1(\mu)} \sim m(\widehat{Q}), \qquad \forall \ i,j \in \{1, \dots, 2^n-1\}, Q \in \D. \]
\end{remark}

We now introduce sparse operators adapted to the complexity of the shifts, and we record the best known weighted inequalities in the nonhomogeneous setting.
 \begin{definition}
     Given a sparse family $\mc{S} \subset \D$, $N=s+t\in \N$ and a locally integrable function $f$, we define the sparse form of complexity $N$ as 
     \begin{equation}
         \mc{A}^N_\mc{S}f(x)=\sum_{Q \in \mc{S}}\ang{f}_Q \1_{Q}(x) +\sum_{\substack{ J, K \in \mc{S} \\ \text{dist}(J, K) \leq N+2}}\ang{f}_{J} \frac{\1_K(x)}{\mu(K)}\sqrt{m(J)}\sqrt{m(K)} .
     \end{equation}
 \end{definition}
 We now define adapted weight classes.
 \begin{definition}
  Let $1<p<\infty$ and $N\in\N$. Given cubes $Q,R\in\mathcal D$, we denote
    \begin{equation*}
        c_p^b(Q,R)=\begin{cases}
            1,\textrm{ if }Q=R,\\
            \frac{m(Q)^{p/2}m(R)^{p/2}}{\mu(R)\mu(Q)^{p-1}},\textrm{ otherwise}.
        \end{cases}
    \end{equation*}
We say that a weight $w \in A_p^N(\mu)$ if 
\[ [w]_{A_p^N(\mu)}:=\sup_{\substack{Q,R\in\mathcal D\\0\leq\mathrm{dist}(Q,R)\leq N+2}}c_p^b(Q,R)\ang{w}_Q \ang{\sigma}_R^{p-1} <\infty. \]
 \end{definition}
 Given two balanced pairs $(\mu, \mathscr{H})$ and $(\mu, \widetilde{\mathscr{H}})$, weighted estimates are equivalent in light of \eqref{equiv of balanced constants}.
Although we define complexity-dependent weight characteristics $[W]_{A_p^N}$, the weight \textit{classes} are the same independent of the complexity, even though quantitative weighted estimates depend exponentially on the complexity. They are all unified under the following condition.
\begin{definition}
    Let $1<p<\infty$. We say that $w\in A_p^b(\mu)$ if
    \begin{equation*}
        \sup_{\substack{Q,R\in\mathcal D\\ R\in\mathrm{ch}(\widehat Q)\cup\mathrm{ch}\left(Q^{(2)}\right)\\\text{or } Q\in\textrm{ch}\left(R^{(2)}\right)}}c_p^b(Q,R)\ang{w}_Q \ang{\sigma}_R^{p-1} < \infty.
    \end{equation*}
    where $Q^{(1)}=\widehat Q$ and $Q^{(j)}=\widehat{Q^{(j-1)}}$ for $j\geq 2$.
\end{definition}
\begin{proposition}[\cite{BBDPW}]\label{prop:stability}
    For $1<p<\infty$ and $N\in\N$, we have
    \begin{equation*}
    [w]_{A_p^b(\mu)}\leq[w]_{A_p^N(\mu)}\lesssim\left([w]_{A_p^b(\mu)}\right)^{2^{N-1}}.
    \end{equation*}
    In particular, $A_p^N(\mu)=A_p^M(\mu)$ for all $N,M\in\N$.
\end{proposition}

\begin{theorem}[\textup{\cite[Theorem A and Corollary 1.2]{BBDPW}}] \label{sparseforshifts}
Let $\mu$ be an atomless Radon measure in $\R^n$ and $\mathscr{H}$ a generalized Haar system such that the pair $(\mu, \mathscr{H})$ is balanced. Let $f \in L^1(\R^n)$ be compactly supported in $Q_0 \in \D$, and $T$ be a generalized Haar shift of complexity $(s,t)$ as in Definition \ref{def: vectorhaarshift}, with $N=s+t\in \N$. There exists a sparse family $\mathcal{S}= \mathcal{S}(f)\subset \mathcal{D}(Q_0)$ and a positive constant $C=C(n,N,T, \mu, \mathscr{H})$, depending exponentially on the complexity, such that  
\[|Tf(x)| \leq C \mc{A}^N_\mc{S}(|f|)(x) \quad \text{ on } Q_0.\]   
Consequently, if $1<p<\infty$ and $w \in A_p^b(\mu)$ there holds
$$ \|T  \|_{L^p(w) \to L^p(w)} \lesssim [w]_{A^\D_p}^{1 + \frac{1}{p-1} - \frac{2}{p}}[w]_{A_p^N}^{\frac{1}{p}} \lesssim [w]_{A^\D_p}^{1 + \frac{1}{p-1} - \frac{2}{p}}[w]_{A_p^b(\mu)}^{\frac{2^{N-1}}{p}},$$
where the implicit constant depends only on $n, N, p, \mu$ and $\mathscr{H}$. \par 
If $\mu$ is a general Radon measure and $T$ is $L^1$ normalized as in \cref{D: L^1 normalized}, for each $f \in L^1(\R^n;\R^d)$ compactly supported in $Q_0 \in \D$, there exists a sparse family $\mathcal{S}= \mathcal{S}(f)\subset \mathcal{D}(Q_0)$ and a positive constant $C=C(n,N,T)$ depending linearly on the complexity such that

\[|Tf(x)| \leq C \mc{A}_\mc{S}|f|(x) \quad \text{ on } Q_0. \]
Consequently, for every $1<p<\infty$ $w \in A_p^\D(\mu)$ we have \[ \|T\|_{L^p(w) \to L^p(w)} \lesssim_{p,d} [w]_{A_p^\D(\mu)}^{\max \big(1, \frac{1}{p-1}\big)}.\]
\end{theorem}
\begin{remark} The previous result was stated in the vector valued setting in \cite{BBDPW}, but the convex body domination argument given recovers pointwise sparse domination in the scalar setting. As we have also seen in the proof of \cref{thm:SparseDom}, sparse domination results for dyadic operators revolve around estimating $T_{\widehat{Q}}f(x)\1_Q(x)$, where $Q$ is a selected cube in the sparse algorithm. In general, it is not possible to control this term with $\ang{|f|}_{\widehat{Q}}$ if $T$ is a Haar shift, and one needs to encompass the complexity of the operator in the modified sparse form, unless the shift is $L^1$ normalized.

For the same reason,  when $N=0$ the result does not recover the usual sparse domination: in the non-homogeneous setting a Haar multiplier $\widetilde{T}$, seen as a zero-complexity operator from Definition \ref{D: dyadic shifts}, is essentially different from a martingale transform of the form \[Tf(x)= \sum_{Q \in \D} c_Q \Delta_Q f(x),\]
which in turn admits usual sparse domination. Indeed, for a martingale transform one has \[|c_Q \Delta_{\widehat{Q}}f(x)\1_Q(x)| \leq \ang{|f|}_{Q}+\ang{|f|}_{\widehat{Q}}\] and the second term is then controlled by the stopping time condition. A similar argument does not work in $\R^n$ for operators as \[\widetilde{T}f(x)= \sum_{Q \in \D} c_Q\ang {f,h_Q}h_Q(x) \]
unless $n=1$ when the two operators coincide.
\end{remark}

\subsection{Improved weighted inequalities for commutators}
We first recall the known weighted inequalities for commutators. The weight class $\widehat{A}_p$ was introduced in \cite{BCAPW} to characterize martingale BMO and to provide a condition that would guarantee a reverse H\"{o}lder inequality. 

\begin{definition}
Let $1<p<\infty$. We say $w \in \widehat{A}_p$ if 
\begin{equation*}
[w]_{\widehat{A}_p(\mu)} := \sup_{\substack{Q \in \D: \\ R \in \{\widehat{Q}, Q, \mathrm{ch}(Q) \}}} \langle w \rangle_{Q} \langle \sigma \rangle_{R}^{p-1} < \infty.
\end{equation*}
\end{definition}
Notice that the argument given in \cite[Proposition 3.6]{BCAPW} adapted to the higher dimensional case $n>1$ yields the estimate $[w]_{A_p^b(\mu)} \lesssim  [w]^4_{\widehat{A_p}(\mu)}$. The following theorem was proved for this weight class:

\begin{theorem} [\cite{BCAPW}]\label{OldWeightedCommutator}
Let $1<p<\infty$, $b \in \BMO$ and $w \in \widehat{A}_p$. Then if $T$ is a generalized Haar shift of complexity $(s,t)$ and $(\mu, \mathscr{H})$ is balanced, then there exists a positive constant $C=C(p,[w]_{\widehat{A}_p},n,N,\mu)$ such that for all $f \in L^p(w)$
$$ \|[T,b]f\|_{L^p(w)} \leq C \|b\|_{\BMO} \|f\|_{L^p(w)}.$$
\end{theorem}
Notice that the proof in \cite{BCAPW} appears in the special case $n=1$, but it can be generalized to every $n \geq 1$ by properly defining balanced pairs as before.
The argument relies on the reverse H\"{o}lder inequality of $w \in \widehat{A_p}(\mu)$ to implement the Cauchy integral trick, while a weight which is merely in the $A^b_p$ class does not have this property. However, using the sparse domination for both Haar shifts and paraproduct forms, we can still deduce weighted inequalities without requiring this property. We now restate and prove \cref{cor:B} as a consequence of the previous estimates.
\begin{theorem} \label{sharp commutator bounds}Suppose $(\mu,\mathscr{H})$ is balanced and $\mu$ is atomless.
Let $1<p<\infty$, $b \in \BMO$, $w \in A_p^{b}(\mu)$ and $T$ a Haar shift of complexity $(s,t)$ with $s+t=N$. Then there exists a positive constant $C=C(p,N,\mu, \mathscr{H},T)$ depending exponentially on $N$ such that for all $f \in L^p(w)$
\begin{equation} \|[T,b]f\|_{L^p(w)} \leq C [w]_{A_p^\D}^{\big(1 + \frac{1}{p-1} - \frac{2}{p}+\max \big(1, \frac{1}{p-1} \big)\big)}[w]_{A_p^b(\mu)}^{\frac{2^{N-1}}{p}}\|b\|_{\BMO} \|f\|_{L^p(w)}. 
\label{ImprovedWeightedCommutator}\end{equation} 
Moreover, if $\mu$ is a Radon measure and $T$ is $L^1$ normalized as in \cref{D: L^1 normalized} we have 
\begin{equation} \|[T,b]f\|_{L^p(w)} \leq C[w]_{A_p^\D(\mu)}^{2\max \big(1, \frac{1}{p-1}\big)}\|b\|_{\BMO} \|f\|_{L^p(w)}. 
\label{WeightedCommutator L1 norm }\end{equation} 
\end{theorem}

\begin{proof} Decompose the commutator as
\begin{equation*}
[T,b]f= [T, \Pi_b]f+[T, \Delta_b]f+ [T, \Lambda_b^0]f.    
\end{equation*}
Notice that, if $T$ is a Haar shift of complexity $(s,t)$, the third term on the right hand side is a Haar shift with at most the same complexity, whose coefficients are bounded by $\|b\|_\BMO$, so the weighted estimates are the same as the weighted estimates for Haar shifts. We refer the reader to \cite{BCAPW} for the computation of the last commutator in the one-dimensional case. 
For the first term, simply write \[\|[T, \Pi_b]\|_{L^p(w) \to L^p(w)}\leq 2\|T\|_{L^p(w) \to L^p(w)}\|\Pi_b\|_{L^p(w) \to L^p(w)},\]
and same holds for the second term. Combining weighted estimates from \cref{sparseforshifts} and \cref{thm:A} yields the result. \end{proof}

\section{Dyadic Hilbert Transform: refined commutator bounds} \label{Section 4} 
In this section we focus on the case $n=1$ and $T= \Hilb$, where the dyadic Hilbert transform $\Hilb$ is defined by its action on Haar functions
\begin{equation}
\Hilb(h_Q)= \text{sign}(Q) h_{Q^s}, \quad Q \in \D.
\end{equation}
Here $h_Q$ is the Haar function associated to $Q$ and adapted to the measure $\mu$, defined as 
\[h_Q(x):= \sqrt{m(Q)}\left(\frac{\1_{Q_+}(x)}{\mu(Q_+)} -\frac{\1_{Q_-}(x)}{\mu(Q_-)}\right); \qquad m(Q):= \frac{\mu(Q_+) \mu(Q_{-})}{\mu(Q)}. \]
The class of measures for which $\Hilb$ extends to a bounded operator on $L^p(\mu)$ is in general strictly larger than the balanced class.
\begin{proposition} [\textup{\cite[Proposition 1.2]{BCAPW}}]
The following are equivalent.
\begin{enumerate}
    \item $\Hilb$ is bounded on $L^p(\mu)$ for all $1<p<\infty$;
    \item $\Hilb$ is bounded on $L^p(\mu)$ for some $p \neq 2$;
    \item $\mu$ is sibling balanced, which means
    $$ [\mu_{sib}]:= \sup_{Q \in \D} \frac{m(Q)}{m(Q^s)}<\infty.$$
    \item $\Hilb$ is weak-type $(1,1).$
\end{enumerate}
\end{proposition}
In the same spirit, if one is concerned with $L^p(w)$ estimates for the operator $\Hilb$ \emph{alone}, one can assume a weaker condition on the weight $w$ than what assumed before, and \cref{sharp commutator bounds} allows us to get sharper weighted inequalities.

\begin{definition}[\textup{\cite[Appendix A.2]{BCAPW}}] \label{SibBalancedWeights}
Let $1<p<\infty$. A weight $w \in A_p^{sib}(\mu)$ if 
    \begin{equation*}
        [w]_{A_{p}^{sib}(\mu)}:=\sup_{Q, R\in \D}c_{p}(Q,R)\langle w\rangle_{Q} \langle \sigma\rangle_R^{p-1} <\infty,
    \end{equation*}
    where 
    \begin{equation*}\label{defconstants}
        c_p(Q,R)=\begin{cases}
            1, \text{ if }Q=R,\\
              \left(\frac{m(\widehat{Q})}{\mu(R)}\right)^{p-1}\frac{m(\widehat{R})}{\mu(R)} , \text{ if }\widehat{Q}=(\widehat{R})^s,\\
              \left(\frac{m(Q)}{\mu(Q)}\right)^{p-1}\frac{m(\widehat{R})}{\mu(R)} , \text{ if }Q=(\widehat{R})^s,\\
             0, \text{ for any other case}.
        \end{cases}
    \end{equation*}
    
\end{definition}

Even assuming that the measure is merely sibling balanced, $\Hilb$ still admits a modified sparse domination. 
If $\mathcal{S}$ is a sparse family and $f \in L^\infty_{\mathrm{loc}}$, we define
\begin{equation*}
         \begin{split}
\mathcal{E}^{\mathcal{S}}_1(f)(x):=&\sum_{\substack{Q,R \in \mathcal{S}\\ \widehat{Q}=(\widehat{R})^{s}}} \langle f \rangle_Q \,\frac{m(\widehat{Q})^{1/2} m(\widehat{R})^{1/2}}{\mu(R)} \textbf{1}_R(x),\\
\mathcal{E}^{\mathcal{S}}_2(f)(x):=&\sum_{\substack{Q,R\in \mathcal{S}\\ {Q}=(\widehat{R})^{s}}}\ \langle f\rangle_Q  \,\frac{m(Q)^{1/2} m(\widehat{R})^{1/2}}{\mu(R)} \textbf{1}_R(x),\\
\mathcal{E}_{\mathcal{S}}(f)(x):=& \mathcal{A}_{\mathcal{S}}(f)(x)+\sum_{j=1}^2 \mathcal{E}^{\mathcal{S}}_j(f)(x).
         \end{split}
    \end{equation*}

\begin{remark}
The careful reader will notice that in Definition \ref{SibBalancedWeights}, the configuration $(Q,R)$ of intervals satisfying $R=(\widehat{Q})^s$ has been removed. This symmetrization is unavoidable in the bilinear setting, where stopping conditions are imposed on two functions simultaneously. It does not arise, however, if one runs the pointwise sparse domination argument via weak-type estimates. One needs to control a term like $\langle f, h_Q \rangle h_{Q^s}$, and there is never a need to replace the characteristic functions $\mathbf{1}_{(Q^s)_{-}}$ and $\mathbf{1}_{(Q^s)_{+}}$ by the characteristic function of the parent interval. Therefore, the assumption on the weight class can actually be slightly weakened from the version in \cite{BCAPW}.
    
\end{remark}
\begin{theorem}[\textup{\cite[Theorem A.2]{BCAPW}}]\label{sparseresultforH}
    If $\mu$ is sibling balanced and atomless, there exists $\eta\in (0,1) $ such that for each $L^1$ function $f$ compactly supported on $Q_0 \in \D$,  there exists an $\eta$-sparse collection $\mathcal{S}\subset \D$ such that for $\mu$ a.e. $x \in Q_0$,
    $$|\Hilb f (x)|\lesssim \mathcal{E}_{\mathcal{S}}(|f|)(x).
    $$
    Moreover, for $1<p<\infty$, any $\eta$-sparse collection $\mathcal{S}$, $w \in A_p^{sib}(\mu)$, there exists $C=C(p, \mu, \Hilb)$ such that for any $f \in L^p(w)$
$$ \|\mathcal{E}_{\mathcal{S}}(|f|)\|_{L^p(w)} \leq C(p)[w]_{A_p}^{1 + \frac{1}{p-1} - \frac{2}{p}}[w]_{A_p^{sib}}^{\frac{1}{p}} \|f\|_{L^p(w)}.$$
\end{theorem}

As before, the result was stated in the bilinear sense in \cite{BCAPW} but can be improved to a pointwise sparse domination.
\begin{corollary} Suppose $\mu$ is sibling balanced and atomless.
Let $1<p<\infty$, $b \in \BMO$ and $w \in A_p^{sib}$. Then there exists a constant $C=C(p,\mu, \Hilb)>0$ such that for all $f \in L^p(w)$
\begin{equation}\|[\mathcal{H},b]f\|_{L^p(w)} \leq C(p) [w]_{A^\D_p(\mu)}^{\big(1 + \frac{1}{p-1} - \frac{2}{p}+\max \big(1, \frac{1}{p-1} \big)\big)}[w]_{A_p^{sib}(\mu)}^{\frac{1}{p}}\|b\|_{\BMO} \|f\|_{L^p(w)}. \label{ImprovedWeightedCommutator2}\end{equation} 
\end{corollary}

The next subsections are concerned with proving \cref{thm:D}.

\subsection{$L^p$ boundedness of $[\Hilb,b]$: necessary and sufficient conditions.} 

Define for $1<p<\infty$ \[ {[\BMO]}_p(\mu):= \{ b \in \bmo_p(\mu): \|[b,\Hilb]\|_{L^p(\mu) \to L^p(\mu)}< \infty \}.\]

The following has been proved in \cite{BCAPW}.
\begin{theorem}
Let $\mu$ be a sibling balanced measure, $1<p<\infty$ and $b \in \BMO(\mu)$. Then 
\[\|[\Hilb,b]\|_{L^p(\mu) \to L^p(\mu)} \lesssim \|b\|_\BMO.\]
Moreover, we have that 
\begin{equation} \label{CP and comm norm}  \|b\|_{\bmo_p} \leq \|[b, \Hilb]\|_{L^p(\mu) \to L^p(\mu)}. \end{equation}
\end{theorem}
The previous theorem says that \begin{equation} \label{containment of bmo spaces}\BMO(\mu) \subseteq [\BMO]_p(\mu) \subseteq \bmo_p(\mu), \quad 1<p<\infty.\end{equation}
We now give a precise characterization of $[\BMO]_p(\mu).$
\begin{theorem}\label{T: nec and suff cond}

Let $b$ be locally integrable, $1 < p < \infty $, and $\mu$ sibling balanced. The commutator $[\Hilb, b]$ extends to a bounded operator on $L^p(\mu)$ if and only if the following conditions are satisfied:

\begin{enumerate}
    \item \label{Nec 1} The symbol $b \in \bmo_{\alpha(p)}(\mu)$, where $\alpha(p)=\max(p,p')$;
    \item \label{Nec 2} The sequence $\beta=\{\beta_Q\}_{Q \in \D}$ with $\beta_Q= c_Q-c_{Q^s}$ and $c_Q= \langle b, h_Q^2 \rangle$ satisfies $\|\beta\|_{\ell^{\infty}}<\infty$.
\end{enumerate}
In other words for $1< p < \infty$ and $\alpha(p):=\max(p,p')$
\[ [\BMO]_p(\mu)=\{b \in \bmo_{\alpha(p)}(\mu), \beta \in \ell^\infty \}.\]

\end{theorem}
\begin{remark}
 In the case $p=2$ the first condition in \cref{T: nec and suff cond} is the usual Carleson condition. Also, if the measure $\mu$ is dyadically doubling, it is easy to see that this condition implies \ref{Nec 2}. Indeed, \ref{Nec 1} implies $ \sup_{Q \in \D} \| \Delta_Q b \|_\infty < \infty$ and $h_Q^2(x) \sim \frac{\1_Q(x)}{\mu(Q)}$, so
\[|\beta_Q| \sim | \ang{b}_Q - \ang{b}_{Q^s}| + |\ang{b}_{Q_{-}}-\ang{b}_{Q_+}|+ \ang{b}_{Q_{-}^s}-\ang{b}_{Q_+^s}| \leq 3 \sup_{Q} \|\Delta_Qb \|_\infty < \infty.\]
\end{remark}
\begin{proof}
Use the splitting of the commutator
$$[\Hilb, b]=[\Hilb, \Pi_b]+[\Hilb, \Pi^*_b]+[\mathcal{H},\Lambda_b],$$
where $\Pi_b^*$ denotes the formal adjoint of the paraproduct $\Pi_b$ and $$\Lambda_b(f)= \sum_Q \Delta_Q(b \Delta_Qf)=\sum_Q c_Q \langle f, h_Q\rangle h_Q, \quad c_Q:= \langle b, h_Q^2\rangle$$
is a martingale multiplier. Let's prove the sufficiency first. \par 
Recall that $\Pi_b$ is bounded on $L^p(\mu)$ if and only if $b \in \bmo_p(\mu)$ by \cref{L^p bound of paraprod}. %Also, since $\Hilb^*=-\Hilb$, then $[\Hilb, \Pi_b^*]$ is the formal adjoint of $[\Hilb, \Pi_b]$. 
Hence, if $b \in \bmo_p(\mu) \cap \bmo_{p'}(\mu)$ then $\Pi_b$, $\Pi_b^*$ are both bounded on $L^p(\mu)$, so $[\Hilb, \Pi_b]$, $[\Hilb, \Pi_b^*]$ are both bounded on $L^p(\mu)$ for $1<p<\infty$. Notice that \begin{equation}\label{E: testing on haar} [\Hilb, \Lambda_b](h_Q)(x)=(c_Q-c_{Q^s})h_{Q^s}(x)=: \beta_Qh_{Q^s}(x),\end{equation}
so if $\beta \in \ell^\infty$ also $[\Hilb, \Lambda_b]$ is bounded on $L^p(\mu)$ for $1<p<\infty$. In particular \begin{enumerate}[label=(\roman*)]
    \item \label{Suff 1}for $1 <p \leq 2$, we have $\|b\|_{\bmo_{p}} \leq \|b\|_{\bmo_{p'}}$, hence $b \in \bmo_{p'}(\mu)$ and $\beta \in \ell^\infty$ are sufficient conditions for $L^p$ boundedness of $[\Hilb,b]$;
    \item \label{suff 2}for $2 \leq p < \infty $, we have $\|b\|_{\bmo_{p'}} \leq \|b\|_{\bmo_{p}}$ hence $b \in \bmo_p(\mu)$ and $\beta \in \ell^\infty$ are sufficient for $L^p$ boundedness of $[\Hilb,b]$. 
\end{enumerate}

Conversely, suppose that $[\Hilb,b]$ is bounded on $L^p(\mu)$ for some $1<p<\infty$. It follows that $b \in \bmo_p(\mu)$ by \eqref{CP and comm norm} and that $\Pi_b$ is bounded on $L^p(\mu)$ by \cref{L^p bound of paraprod}. Also, using $\Hilb^*=-\Hilb $ \[\| [\Hilb, b] \|_{L^p(\mu) \to L^p(\mu)}= \| [\Hilb,b]^* \|_{L^{p'}(\mu) \to L^{p'}(\mu)}=\| [\Hilb,b] \|_{L^{p'}(\mu) \to L^{p'}(\mu)} < \infty,\] which in turn implies that $b \in \bmo_{p'}$ and that $\Pi_b$ is bounded on $L^{p'}(\mu)$. Altogether, this implies that $b \in \bmo_p(\mu) \cap \bmo_{p'}(\mu)$ and that $[\Hilb, \Pi_b]$, $[\Hilb, \Pi_b^*]$ are both bounded on $L^p(\mu)$ for $1<p<\infty$, so $[\Hilb, \Lambda_b]$ has to be bounded on $L^p(\mu)$. By \eqref{E: testing on haar} and the fact that $\mu$ is sibling balanced it follows that $\beta \in \ell^\infty.$ We conclude that \ref{Suff 1} and \ref{suff 2} are also necessary respectively when $1<p\leq 2 $ and $2 \leq p <\infty$.
\end{proof}
In particular, the inclusions in \eqref{containment of bmo spaces} are strict. 
\begin{theorem}\label{THM: sharp containment}
 There exists a sibling balanced measure $\mu$ such that the following holds:
\begin{enumerate}
    \item for every $1<p<\infty$ there exists $f_p \in \bmo_p(\mu)$ such that $[\Hilb,f_p]$ is not bounded on $L^p(\mu)$;
    \item there exists a function $q$ such that for every $1<p<\infty$ we have that $q \in \bmo_p(\mu) \setminus \BMO(\mu)$ and $[\Hilb,b_p]$ is bounded on $L^p(\mu)$.
\end{enumerate}   
In other words we have that for every $1<p<\infty$
\[ \BMO(\mu) \subsetneq [\BMO]_p(\mu) \subsetneq \bmo_p(\mu).\]
\end{theorem} 
Before proving this result, we state some corollaries. First of all, note that \cref{T: nec and suff cond} gives  \([\BMO]_p(\mu)= [\BMO]_{p'}(\mu)\) for every \(1<p<\infty,\) so we can restrict to the case $p \geq 2.$
Let \[B(\mu):= \{ b \in L_{\text{loc}}^2(\mu): \beta(b)= (\beta_Q(b))_Q \in \ell^\infty \}\]
where $\beta$ is as in \cref{T: nec and suff cond}. Since for every $p \geq 2$, $[\BMO]_p(\mu)= B(\mu) \cap \bmo_p(\mu)$, using the relation of $\bmo$ norms for $q>p \geq 2$ we get 
\([\BMO]_q(\mu) \subsetneq [\BMO]_p(\mu) \subsetneq [\BMO]_2(\mu),\)
so that $$[\BMO]_2(\mu)=B(\mu) \cap \bmo_2(\mu)= \bigcup_{p \geq 2} [\BMO]_p(\mu).$$  \begin{corollary}\label{commutator symbols} Define
\[[\BMO]_\infty(\mu):= \{b \in [\BMO]_2(\mu): \|[\Hilb,b]\|_{L^p(\mu) \to L^p(\mu)} < \infty, \text{ for every $1<p<\infty$} \},\]
Then we have $\BMO(\mu) \subsetneq [\BMO]_\infty(\mu)$ and \[ [\BMO]_\infty(\mu)= B(\mu) \cap \bigcap_{p \geq 2} \bmo_p(\mu). \]  
\end{corollary}
The fact that the inclusion is strict will also be proved in the following section.
\subsection{Proof of \cref{THM: sharp containment}}
The scheme below constructs an absolutely continuous measure for which \cref{THM: sharp containment} holds. A similar strategy could be employed to construct an atomic measure satisfying the same properties.

For $k\ge1$ define
\[a_k=\begin{cases}1/2,&k=1,\\[2mm]
1/\sqrt{k},&k\ge2,\end{cases}\qquad b_k=1-a_k.\] Let also $c_{kj}=1-\frac{1}{k+j}$ and $d_{kj}=\frac{1}{k+j}$ for $k,j\geq 1$. Set $I=I_0:=[0,1)$ and, for every $n\in \mathbb{Z}$ and $k\geq 1$, define 
\begin{gather*}
    I_k=I_k^1:=[0,2^{-k}),\qquad 
    I_k^b=(I_k^1)^b:=[2^{-k}, 2^{-k+1})\\
    I_{kj}=I_{kj}^1:=[2^{-k},2^{-k}+2^{-k-j}),\qquad
    I_{kj}^b=(I_{kj}^1)^b:=[2^{-k}+2^{-k-j},2^{-k}+2^{-k-j+1}).
    \end{gather*}
In other words, $I_k^b$ is the dyadic sibling of $I_k$, which corresponds to its complement in $I_{k-1}$, and $I_{kj}, I_{jk}^b$ are sibling intervals at scale $j+k$ at the left endpoint of $I_k^b$.
For each $J \in \{I,I_k,I_k^b,I_{kj}, I_{kj}^b\}$, define its integer translation $J^n=J+(n-1)$.

For each $k\geq 1$, we define a function $g^k$ that is supported on $I_k^b$.
\begin{align*}
    g^k(x)&:=\begin{cases}
        0, & x\notin I_k^b\\
        (\prod^{k-1}_{i=1}a_i)b_k(\prod_{i=1}^{j-1}c_{ki})d_{kj}2^{k+j}, &x\in I_{kj}^b. 
    \end{cases}
\end{align*}
Since $\{I_k^b\}_k$ is a partition of $[0,1)$, we define $g$ as the infinite sum of $g^k$ and use $g$ to define an absolutely continuous measure $\mu$ as follows
\begin{align*}
    g(x)&:=\begin{cases}
        0, 
        &x \notin [0,1)\\
        g^k(x), & x\in I_k^b
    \end{cases}\\
    d\mu&:=\sum_{n\in \mathbb{Z}}g(x-n)dx.
\end{align*}
Therefore, $g(x)dx$ is a measure supported on $[0,1)$, and $\mu$ is constructed by periodically translating $g(x)dx$ into intervals of the form $[n-1,n)$. Notice that the measure $\mu$ is always uniform in $I_{kj}^b$.

\begin{figure}
\centering

\resizebox{1.0\textwidth}{!}{

\tikzset{every picture/.style={line width=0.75pt}} %set default line width to 0.75pt        

\begin{tikzpicture}[x=0.75pt,y=0.75pt,yscale=-1.2,xscale=1.2]
%uncomment if require: \path (0,343); %set diagram left start at 0, and has height of 343

%Straight Lines [id:da41053566203426994] 
\draw    (252.24,83.04) -- (453.75,83.04) ;
%Straight Lines [id:da17760047546428137] 
\draw    (252.2,108.93) -- (347.17,108.93) ;
%Straight Lines [id:da9555414144609622] 
\draw    (456.77,83.08) -- (658.28,83.08) ;
%Straight Lines [id:da27251719658402374] 
\draw    (352.87,108.93) -- (453.53,108.93) ;
%Straight Lines [id:da5007756917938001] 
\draw    (457.59,108.93) -- (552.56,108.93) ;
%Straight Lines [id:da3565802952151559] 
\draw    (558.26,108.93) -- (658.93,108.93) ;
%Straight Lines [id:da7959518173263903] 
\draw    (251.33,138.7) -- (295.39,138.7) ;
%Straight Lines [id:da2963278394409622] 
\draw    (299.36,138.7) -- (346.74,138.7) ;
%Straight Lines [id:da6920844048689708] 
\draw    (352.3,138.7) -- (396.36,138.7) ;
%Straight Lines [id:da02827140184198096] 
\draw    (400.33,138.7) -- (452.45,138.49) ;
%Straight Lines [id:da41581149216574975] 
\draw    (456.73,138.27) -- (500.78,138.27) ;
%Straight Lines [id:da9965361094713968] 
\draw    (504.75,138.27) -- (552.13,138.27) ;
%Straight Lines [id:da5450030709243521] 
\draw    (557.7,138.27) -- (601.75,138.27) ;
%Straight Lines [id:da38398028483264524] 
\draw    (605.72,138.27) -- (657.85,138.06) ;
%Straight Lines [id:da9514278977464663] 
\draw    (252.2,168.91) -- (270.79,168.91) ;
%Straight Lines [id:da7244142980766796] 
\draw    (275.2,168.91) -- (294.53,168.91) ;
%Straight Lines [id:da5678992728775413] 
\draw    (301.39,168.91) -- (319.98,168.91) ;
%Straight Lines [id:da21393412510645748] 
\draw    (323.96,168.91) -- (343.29,168.91) ;
%Straight Lines [id:da432637317497937] 
\draw    (352.74,168.91) -- (371.33,168.91) ;
%Straight Lines [id:da5242720862936486] 
\draw    (375.73,168.91) -- (395.07,168.91) ;
%Straight Lines [id:da5326997491803607] 
\draw    (401.93,168.91) -- (420.52,168.91) ;
%Straight Lines [id:da26356760275778346] 
\draw    (424.93,168.91) -- (449.87,168.91) ;
%Straight Lines [id:da003789913064220163] 
\draw    (459.75,168.91) -- (478.34,168.91) ;
%Straight Lines [id:da8290579610411417] 
\draw    (482.75,168.91) -- (502.08,168.91) ;
%Straight Lines [id:da7314470333607429] 
\draw    (508.94,168.91) -- (527.54,168.91) ;
%Straight Lines [id:da7655990927073788] 
\draw    (531.94,168.91) -- (551.27,168.91) ;
%Straight Lines [id:da25096485535466084] 
\draw    (560.72,168.91) -- (579.32,168.91) ;
%Straight Lines [id:da4643400695463188] 
\draw    (583.72,168.91) -- (603.05,168.91) ;
%Straight Lines [id:da7240545226859343] 
\draw    (609.91,168.91) -- (628.51,168.91) ;
%Straight Lines [id:da006078339202536975] 
\draw    (632.91,168.91) -- (652.24,168.91) ;

% Text Node
\draw (346.37,67.41) node [anchor=north west][inner sep=0.75pt]  [font=\tiny] [align=left] {$\displaystyle a_{1}$};
% Text Node
\draw (548.31,69.57) node [anchor=north west][inner sep=0.75pt]  [font=\tiny] [align=left] {$\displaystyle b_{1}$};
% Text Node
\draw (292.89,94.17) node [anchor=north west][inner sep=0.75pt]  [font=\tiny] [align=left] {$\displaystyle a_{1} a_{2}$};
% Text Node
\draw (262.56,123.94) node [anchor=north west][inner sep=0.75pt]  [font=\tiny] [align=left] {$\displaystyle a_{1} a_{2} a_{3}$};
% Text Node
\draw (394.29,96.32) node [anchor=north west][inner sep=0.75pt]  [font=\tiny] [align=left] {$\displaystyle a_{1} b_{2}$};
% Text Node
\draw (496.13,97.62) node [anchor=north west][inner sep=0.75pt]  [font=\tiny] [align=left] {$\displaystyle b_{1} c_{11}$};
% Text Node
\draw (600.7,97.62) node [anchor=north west][inner sep=0.75pt]  [font=\tiny] [align=left] {$\displaystyle b_{1} d_{11}$};
% Text Node
\draw (363.4,123.51) node [anchor=north west][inner sep=0.75pt]  [font=\tiny] [align=left] {$\displaystyle a_{1} b_{2} c_{21}$};
% Text Node
\draw (411.44,123.94) node [anchor=north west][inner sep=0.75pt]  [font=\tiny] [align=left] {$\displaystyle a_{1} b_{2} d_{21}$};
% Text Node
\draw (310.46,125.23) node [anchor=north west][inner sep=0.75pt]  [font=\tiny] [align=left] {$\displaystyle a_{1} a_{2} b_{3}$};
% Text Node
\draw (471.1,123.51) node [anchor=north west][inner sep=0.75pt]  [font=\tiny] [align=left] {$\displaystyle b_{1} c_{21}$};
% Text Node
\draw (574.51,121.51) node [anchor=north west][inner sep=0.75pt]  [font=\fontsize{0.33em}{0.4em}\selectfont] [align=left] {$\displaystyle \frac{b_{1} d_{11}}{2}$};
% Text Node
\draw (624.99,121.51) node [anchor=north west][inner sep=0.75pt]  [font=\fontsize{0.33em}{0.4em}\selectfont] [align=left] {$\displaystyle \frac{b_{1} d_{11}}{2}$};
% Text Node
\draw (519.58,122.64) node [anchor=north west][inner sep=0.75pt]  [font=\tiny] [align=left] {$\displaystyle b_{1} d_{21}$};
% Text Node
\draw (251.33,159.31) node [anchor=north west][inner sep=0.75pt]  [font=\fontsize{0.33em}{0.4em}\selectfont] [align=left] {$\displaystyle a_{1} a_{2} a_{3} a_{4}$};
% Text Node
\draw (275.93,177.43) node [anchor=north west][inner sep=0.75pt]  [font=\fontsize{0.33em}{0.4em}\selectfont] [align=left] {$\displaystyle a_{1} a_{2} a_{3} b_{4}$};
% Text Node
\draw (300.82,158.88) node [anchor=north west][inner sep=0.75pt]  [font=\fontsize{0.33em}{0.4em}\selectfont] [align=left] {$\displaystyle a_{1} a_{2} b_{3} c_{31}$};
% Text Node
\draw (322.83,176.57) node [anchor=north west][inner sep=0.75pt]  [font=\fontsize{0.33em}{0.4em}\selectfont] [align=left] {$\displaystyle a_{1} a_{2} b_{3} d_{31}$};
% Text Node
\draw (350.31,158.02) node [anchor=north west][inner sep=0.75pt]  [font=\fontsize{0.33em}{0.4em}\selectfont] [align=left] {$\displaystyle a_{1} b_{2} c_{21} c_{22}$};
% Text Node
\draw (373.61,176.14) node [anchor=north west][inner sep=0.75pt]  [font=\fontsize{0.33em}{0.4em}\selectfont] [align=left] {$\displaystyle a_{1} b_{2} c_{21} d_{22}$};
% Text Node
\draw (404.37,149.99) node [anchor=north west][inner sep=0.75pt]  [font=\fontsize{0.33em}{0.4em}\selectfont] [align=left] {$\displaystyle \frac{a_{1} b_{2} d_{21}}{2}$};
% Text Node
\draw (428.53,177.2) node [anchor=north west][inner sep=0.75pt]  [font=\fontsize{0.33em}{0.4em}\selectfont] [align=left] {$\displaystyle \frac{a_{1} b_{2} d_{21}}{2}$};
% Text Node
\draw (562.42,151.28) node [anchor=north west][inner sep=0.75pt]  [font=\fontsize{0.33em}{0.4em}\selectfont] [align=left] {$\displaystyle \frac{b_{1} d_{11}}{4}$};
% Text Node
\draw (587.7,177.2) node [anchor=north west][inner sep=0.75pt]  [font=\fontsize{0.33em}{0.4em}\selectfont] [align=left] {$\displaystyle \frac{b_{1} d_{11}}{4}$};
% Text Node
\draw (636.64,177.2) node [anchor=north west][inner sep=0.75pt]  [font=\fontsize{0.33em}{0.4em}\selectfont] [align=left] {$\displaystyle \frac{b_{1} d_{11}}{4}$};
% Text Node
\draw (613.34,151.28) node [anchor=north west][inner sep=0.75pt]  [font=\fontsize{0.33em}{0.4em}\selectfont] [align=left] {$\displaystyle \frac{b_{1} d_{11}}{4}$};
% Text Node
\draw (510.21,151.28) node [anchor=north west][inner sep=0.75pt]  [font=\fontsize{0.33em}{0.4em}\selectfont] [align=left] {$\displaystyle \frac{b_{1} d_{21}}{2}$};
% Text Node
\draw (536.1,177.17) node [anchor=north west][inner sep=0.75pt]  [font=\fontsize{0.33em}{0.4em}\selectfont] [align=left] {$\displaystyle \frac{b_{1} d_{21}}{2}$};
% Text Node
\draw (458.74,158.45) node [anchor=north west][inner sep=0.75pt]  [font=\fontsize{0.33em}{0.4em}\selectfont] [align=left] {$\displaystyle b_{1} c_{21} c_{22}$};
% Text Node
\draw (483.77,177.43) node [anchor=north west][inner sep=0.75pt]  [font=\fontsize{0.33em}{0.4em}\selectfont] [align=left] {$\displaystyle b_{1} c_{21} d_{22}$};

\end{tikzpicture}
}
\caption{The construction of $\mu$ on $ [0,1)$}
\end{figure}

We can calculate the measure of $\mu$ for $I_{kj}^b$ and $I_k^b$.
\begin{align*}
    \mu(I_{kj}^b)&=\int_{I_{kj}^b}(\prod^{k-1}_{i=1}a_i)b_k(\prod_{i=1}^{j-1}c_{ki})d_{kj}2^{k+j}dx=(\prod^{k-1}_{i=1}a_i)b_k(\prod_{i=1}^{j-1}c_{ki})d_{kj}\\
    \mu(I_k^b)&=\sum_{j=1}^\infty \mu(I^b_{kj})=\sum_{j=1}^\infty (\prod^{k-1}_{i=1}a_i)b_k(\prod_{i=1}^{j-1}c_{ki})d_{kj}=(\prod^{k-1}_{i=1}a_i)b_k(\sum_{j=1}^\infty (\prod_{i=1}^{j-1}c_{ki})d_{kj})=(\prod^{k-1}_{i=1}a_i)b_k\\
    \mu([0,1))&=\sum_{k=1}^\infty \mu(I^b_{k})=\sum_{i\geq 1}(\prod_{j=1}^{i-1}a_j)b_i=1.
\end{align*}
The last two equalities can be proved by noticing that the series involved are telescoping.
\begin{proposition}
    $\mu$ is sibling balanced but not balanced.
\end{proposition} \begin{proof}
    Let $I$ be a dyadic interval. By construction of $\mu$ we can restrict to consider $I \subseteq [0,1)$.  For $I_0=[0,1)$ the claim is obvious, as $\mu([0,1))=1$ and $\mu([0,\frac{1}{2}))=a_1=\frac{1}{2}$. When $I \subset [0,1)$ there are two cases:
\begin{enumerate}
\item $\widehat{I}\subset I_k^b$ for some $k\geq 1$. There are two sub-cases. \begin{enumerate}[label=(\roman*)]
    \item $\widehat{I}\subset I_{kj}^b$ for some $j\geq 1$. As $\mu$ is uniform in $I_{kj}^b$, we have $m(I)=m(I^s)$.
    \item $I=I_{kj}$ or $I=I_{kj}^b$. Short calculations reveal that $m(I_{k,j}^b)= \frac{1}{4} d_{kj}\mu(\widehat{I}_{kj})$ and $m(I_{kj})= c_{k(j+1)} d_{k(j+1)} c_{kj} \mu(\widehat{I}_{kj}).$ Therefore, the ratio $\frac{m(I_{kj})}{m(I_{kj}^b)}$ converges to $4$ as $j, k \rightarrow \infty$, and is bounded above and below.

\end{enumerate}
\item $I=I_k$ or $I=I_k^b$. In this case, we compute $m(I_k^b)=c_{k1}d_{k1}b_k\mu(\widehat{I}_k)$, $m(I_k)= a_{k+1}b_{k+1}a_k\mu(\widehat{I}_k)$, and $m(\widehat{I_k})= a_{k}b_{k}\mu(\widehat{I_k}).$ The ratio $\frac{m(I_k)}{m(I_k^b)}$ converges to $1$ and is bounded above and below. The ratio  $\frac{m(I_k)}{m(\widehat{I}_k)}$ converges to $0$, proving $\mu$ is not balanced.

\end{enumerate}
We conclude that $\mu$ is sibling balanced but not balanced. \end{proof}
\begin{proposition}
Let $1< p<\infty$. Consider $d_{kj}$ as above. Define
\begin{align*}
    f_p(x)&:=\begin{cases}d_{(n+1)1}^{-1/p}
        =(n+2)^{1/p}, &x\in (I^n_{(n+1)1})^b, n\geq 1\\
        0,&\text{otherwise}.
    \end{cases}.
\end{align*}
Then
\begin{enumerate}
\item \label{Carleson first example} $\sup_{I\in \mathscr{D}}\frac{1}{\mu(I)}\int_I|f_p-\LL f_p \RR_I|^pd\mu<\infty$. 
\item \label{not in BMO first example}$[\mathcal{H},f_p]$ is not bounded on $L^p$.
\end{enumerate}
Hence, $f_p\in \bmo_p(\mu)\setminus [\BMO]_p(\mu)$ and $[\BMO]_p(\mu)\subsetneq \bmo_p(\mu)$.
\end{proposition}
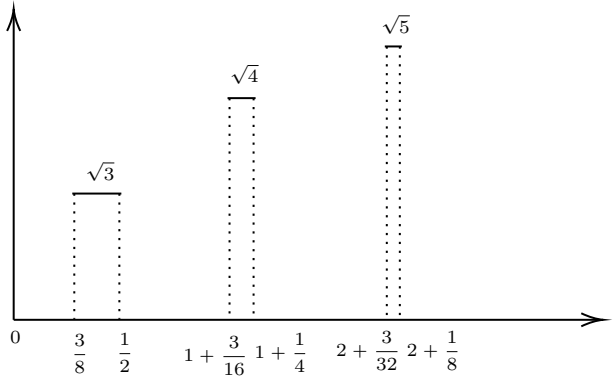
\begin{figure}[h]
\centering

\tikzset{every picture/.style={line width=0.75pt}} %set default line width to 0.75pt        

\begin{tikzpicture}[x=0.75pt,y=0.75pt,yscale=-1,xscale=1]
%uncomment if require: \path (0,300); %set diagram left start at 0, and has height of 300

%Straight Lines [id:da36579249381806944] 
\draw    (145.5,180) -- (437.75,180) ;
\draw [shift={(439.75,180)}, rotate = 180] [color={rgb, 255:red, 0; green, 0; blue, 0 }  ][line width=0.75]    (10.93,-3.29) .. controls (6.95,-1.4) and (3.31,-0.3) .. (0,0) .. controls (3.31,0.3) and (6.95,1.4) .. (10.93,3.29)   ;
%Straight Lines [id:da18178222440424163] 
\draw    (145.5,180) -- (145.5,25.75) ;
\draw [shift={(145.5,23.75)}, rotate = 90] [color={rgb, 255:red, 0; green, 0; blue, 0 }  ][line width=0.75]    (10.93,-3.29) .. controls (6.95,-1.4) and (3.31,-0.3) .. (0,0) .. controls (3.31,0.3) and (6.95,1.4) .. (10.93,3.29)   ;
%Straight Lines [id:da8574045809641836] 
\draw    (174.75,116.5) -- (199.25,116.5) ;
%Straight Lines [id:da20320736157349706] 
\draw    (252.25,68.5) -- (266.25,68.5) ;
%Straight Lines [id:da7660458312638256] 
\draw    (330.75,42.5) -- (339.25,42.5) ;
%Straight Lines [id:da8263869062433769] 
\draw  [dash pattern={on 0.84pt off 2.51pt}]  (175.75,116.5) -- (175.75,179.5) ;
%Straight Lines [id:da21018585295328607] 
\draw  [dash pattern={on 0.84pt off 2.51pt}]  (198.25,116.5) -- (198.25,179.5) ;
%Straight Lines [id:da48086575473370763] 
\draw  [dash pattern={on 0.84pt off 2.51pt}]  (253.25,68.5) -- (253.25,179.5) ;
%Straight Lines [id:da6835587336208365] 
\draw  [dash pattern={on 0.84pt off 2.51pt}]  (265.25,68.5) -- (265.25,179.5) ;
%Straight Lines [id:da09243061185405788] 
\draw  [dash pattern={on 0.84pt off 2.51pt}]  (331.75,42.5) -- (331.75,180) ;
%Straight Lines [id:da6031278782792902] 
\draw  [dash pattern={on 0.84pt off 2.51pt}]  (338.25,42.5) -- (338.25,180) ;

% Text Node
\draw (142.25,184) node [anchor=north west][inner sep=0.75pt]  [font=\tiny] [align=left] {$\displaystyle 0$};
% Text Node
\draw (172.75,185.5) node [anchor=north west][inner sep=0.75pt]  [font=\tiny] [align=left] {$\displaystyle \frac{3}{8}$};
% Text Node
\draw (195.25,185) node [anchor=north west][inner sep=0.75pt]  [font=\tiny] [align=left] {$\displaystyle \frac{1}{2}$};
% Text Node
\draw (264.25,185) node [anchor=north west][inner sep=0.75pt]  [font=\tiny] [align=left] {$\displaystyle 1+\frac{1}{4}$};
% Text Node
\draw (228.75,186.5) node [anchor=north west][inner sep=0.75pt]  [font=\tiny] [align=left] {$\displaystyle 1+\frac{3}{16}$};
% Text Node
\draw (305,184) node [anchor=north west][inner sep=0.75pt]  [font=\tiny] [align=left] {$\displaystyle 2+\frac{3}{32}$};
% Text Node
\draw (340.25,184) node [anchor=north west][inner sep=0.75pt]  [font=\tiny] [align=left] {$\displaystyle 2+\frac{1}{8}$};
% Text Node
\draw (180.25,99) node [anchor=north west][inner sep=0.75pt]  [font=\tiny] [align=left] {$\displaystyle \sqrt{3}$};
% Text Node
\draw (252.25,50) node [anchor=north west][inner sep=0.75pt]  [font=\tiny] [align=left] {$\displaystyle \sqrt{4}$};
% Text Node
\draw (327.75,25) node [anchor=north west][inner sep=0.75pt]  [font=\tiny] [align=left] {$\displaystyle \sqrt{5}$};

\end{tikzpicture}
\caption{A visualization of $f_2$.}
\end{figure}
\begin{proof}
We first prove \ref{Carleson first example}. As $f_p(x)= 0$ when $x<0$, we can restrict to $I\subset[0,\infty)$.
\begin{enumerate}[label=(\roman*)]
\item If $|I|\geq 1$, then $I=[n-1,n-1+m)$ for some $n\geq 1$ and some positive integer $m$. Note that 
\begin{align*}
\lim_{n\rightarrow\infty}\int_{[n-1,n)}f_pd\mu&=\lim_{n\rightarrow\infty}(\prod^{n}_{i=1}a_i)b_{n+1}(d_{(n+1)1})^{1-\frac{1}{p}}=0,
\end{align*}
so using this fact, we estimate the average
\begin{align*}
\LL f_p\RR_{[n-1,n-1+m)}&=\frac{1}{m}\int_{[n-1,n-1+m)}f_pd\mu=\frac{\sum_{i=n}^{n-1+m}\int_{[i-1,i)}f_pd\mu}{m}\lesssim 1.
\end{align*}

In a similar way, one can show 
$$ \lim_{n\rightarrow\infty}\int_{[n-1,n)}f_p^pd\mu=0,$$ which leads to the estimate 
$$\frac{1}{\mu(I)}\int_I|f_p-\LL f_p\RR_I|^pd\mu \lesssim 1.$$

% \lim_{n\rightarrow\infty}\int_{[n-1,n)}f_p^pd\mu&=\lim_{n\rightarrow\infty}(\prod_{i=1}^na_i)b_{n+1}=0.\\
% \LL f_p^p\RR_{[n-1,n-1+m)}&=\frac{1}{m}\int_{[n-1,n-1+m)}f_p^pd\mu=\frac{\sum_{i=n}^{n-1+m}\int_{[i-1,i)}f_p^pd\mu}{m}\lesssim 1.\\
% \frac{1}{\mu(I)}\int_I|f_p-\LL f_p\RR_I|^pd\mu&\leq \frac{1}{\mu(I)}\int_I2^p(f_p^p+\LL f_p\RR_I^p)d\mu=2^p(\LL f_p^p\RR_I+\LL f_p\RR_I^p)\lesssim_p 1
% \end{align*}
In the calculations above, we used the fact that $b_{n+1}, d_{(n+1)1}<1$ and $\lim_n\prod_{i=1}^na_i=0$.
\item If $|I|<1$, then $I$ is strictly contained in some interval $[n-1,n)$ for $n\geq 1$.
If $|I|\leq 2^{-n-2}$ or $I\cap (I^n_{(n+1)1})^b=\emptyset$, then $f$ is constant on $I$ and thus
\[ \frac{1}{\mu(I)}\int_I|f_p-\LL f_p\RR_I|^pd\mu=0.\]
If $|I|>2^{-n-2}$ and $I\cap (I^n_{n+1})^b\ne \emptyset$, then $I$ must contain $(I^n_{n+1})^b=\widehat{(I^n_{(n+1)1})^b}$ and thus $\mu(I)\geq \mu((I^n_{n+1})^b)$. We bound the averages
\begin{align*}
\LL f_p\RR_I&=\frac{1}{\mu(I)}\int_If_pd\mu\leq \frac{1}{\mu((I^n_{n+1})^b)}(d_{(n+1)1}^{-1/p})\mu((I_{(n+1)1}^n)^b)=(d_{(n+1)1})^{1-1/p}\leq 1,\\
\LL f_p^p\RR_I&=\frac{1}{\mu(I)}\int_If_pd\mu\leq \frac{1}{\mu((I^n_{n+1})^b)}(d_{(n+1)1}^{-1})\mu((I_{(n+1)1}^n)^b)=1.
\end{align*}
Putting the above two estimates together, we get
$$\frac{1}{\mu(I)}\int_I|f_p-\LL f_p\RR_I|^pd\mu \leq \frac{1}{\mu(I)}\int_I2^p(f_p^p+\LL f_p\RR_I^p)d\mu=2^p(\LL f_p^p\RR_I+\LL f_p\RR_I^p)\lesssim_p 1.
$$
\end{enumerate}

We are left with \ref{not in BMO first example}. It suffices to show that   \[\sup_I|c_I(f_p)-c_{I^s}(f_p)|=\infty.\]
Notice that $c_I$ can be rewritten as (\cite[page 15]{BCAPW})
\begin{align*}
c_I(f_p)&=\langle f_p, h_I \rangle \int h_I^3 \, d\mu +\langle f_p \rangle_{I}=(\LL f_p\RR_{I_+}-\LL f_p\RR_{I_-})\frac{\mu(I_-)-\mu(I_+)}{\mu(I)}+\LL f_p\RR_I.
\end{align*}
For $I=(I^n_{(n+1)1})^b$ $f_p$ vanishes on $I^s$, so $c_{I^s}(f_p)=0$. As $\mu$ is uniform on $I$ and $f_p$ is constant on $I$ we can conclude that 
$$c_I(f)=\LL f\RR_I=d_{(n+1)1}^{-1/p},$$
\begin{align*}
    \lim_{n\rightarrow\infty}|c_I(f_p)-c_{I^s}(f_p)|&=\lim_{n\rightarrow\infty}d_{(n+1)1}^{-1/p}=\lim_{n\rightarrow\infty}(n+2)^{1/p}=\infty.
\end{align*}
\end{proof}

Define now sequences $(u_k)_{k\ge1}$ and $(v_k)_{k\ge1}$ by
\begin{align*}
v_1&=1,\qquad v_k=v_{k-1}+b_k(-1)^k\log k,\\
u_1&=0,\qquad u_k=v_{k-1}-a_k(-1)^k\log k.
\end{align*}
It is easy to prove that the following properties are satisfied:
    \begin{equation} \label{prop of u,v} a_kv_k+b_ku_k=v_{k-1}, \quad v_k-u_k=(-1)^k\log(k), \quad  \sup_{k \geq 1}|v_ka_k|<\infty.\end{equation}
We now show  that $\BMO(\mu) \subsetneq [\BMO]_p(\mu)$. Define 
\begin{align*}
    p(x)&:=\begin{cases}
        u_k, & x\in I_k^b,k\geq 1\\
        0, &x\notin [0,1)
    \end{cases}
\end{align*}
and $q(x):=\sum_{n\in \mathbb{Z}}p(x-n)$ by periodically translating $p(x)$.
\begin{comment}\begin{proof}
    The first two properties are trivial. We will prove the third one. We decompose $v_ka_k$ into three terms:
   \begin{align*}
    a_k v_k &= \frac{1}{\sqrt{k}} +\frac{1}{\sqrt{k}} \sum_{i=2}^{k-1} (-1)^i \log(i) - \frac{1}{\sqrt{k}} \sum_{i=2}^{k-1} (-1)^i i^{-1/4}.
\end{align*}
The first term clearly goes to $0$. 
The sum $\sum (-1)^i i^{-1/4}$ converges by the Alternating Series Test, so its partial sums are bounded. When multiplied by $k^{-1/2}$, it converges to $0$.
The first sum has an asymptotic behavior of $O(k^{1/4})$. When multiplied by $k^{-1/2}$, the term behaves like $O(k^{-1/4})$, which also converges to $0$, therefore $\sup_k |a_k v_k| < \infty$.
\end{proof}
\end{comment}

\begin{figure}[H]
\centering

\tikzset{every picture/.style={line width=0.75pt}} %set default line width to 0.75pt        

\begin{tikzpicture}[x=0.75pt,y=0.75pt,yscale=-1,xscale=1]
%uncomment if require: \path (0,300); %set diagram left start at 0, and has height of 300

%Straight Lines [id:da874474637626294] 
\draw    (454.4,228.99) -- (690,228.99) ;
%Straight Lines [id:da6550252415045444] 
\draw    (455.4,96) -- (455.4,228.99) ;
%Shape: Brace [id:dp9066016154340356] 
\draw   (456.06,236.01) .. controls (456.06,240.68) and (458.39,243.01) .. (463.06,243.01) -- (501.44,243.01) .. controls (508.11,243.01) and (511.44,245.34) .. (511.44,250.01) .. controls (511.44,245.34) and (514.77,243.01) .. (521.44,243.01)(518.44,243.01) -- (559.81,243.01) .. controls (564.48,243.01) and (566.81,240.68) .. (566.81,236.01) ;
%Straight Lines [id:da1991522895463318] 
\draw    (570.45,141.64) -- (669.97,141.64) ;
%Straight Lines [id:da6308775522883417] 
\draw    (518.87,169.56) -- (566.81,169.56) ;
%Straight Lines [id:da7063462918918179] 
\draw    (481.55,113.12) -- (514.62,113.12) ;
%Shape: Brace [id:dp72905691710434] 
\draw   (456.97,169.56) .. controls (456.97,174.23) and (459.3,176.56) .. (463.97,176.56) -- (478.6,176.56) .. controls (485.27,176.56) and (488.6,178.89) .. (488.6,183.56) .. controls (488.6,178.89) and (491.93,176.56) .. (498.6,176.56)(495.6,176.56) -- (508.84,176.56) .. controls (513.51,176.56) and (515.84,174.23) .. (515.84,169.56) ;
%Shape: Brace [id:dp029797677795955324] 
\draw   (456.47,132.54) .. controls (456.64,135.94) and (458.42,137.56) .. (461.82,137.39) -- (461.82,137.39) .. controls (466.68,137.15) and (469.19,138.73) .. (469.35,142.14) .. controls (469.19,138.73) and (471.54,136.91) .. (476.39,136.68)(474.21,136.78) -- (476.39,136.68) .. controls (479.79,136.51) and (481.41,134.73) .. (481.24,131.33) ;

% Text Node
\draw (496.67,263) node [anchor=north west][inner sep=0.75pt]  [font=\tiny] [align=left] {$\displaystyle \langle q\rangle _{I_{k}} =v_{k}$};
% Text Node
\draw (614.61,126.5) node [anchor=north west][inner sep=0.75pt]  [font=\tiny] [align=left] {$\displaystyle u_{k}$};
% Text Node
\draw (540.3,177.48) node [anchor=north west][inner sep=0.75pt]  [font=\tiny] [align=left] {$\displaystyle u_{k+1}$};
% Text Node
\draw (488.72,119.22) node [anchor=north west][inner sep=0.75pt]  [font=\tiny] [align=left] {$\displaystyle u_{k+2}$};
% Text Node
\draw (467.02,194.47) node [anchor=north west][inner sep=0.75pt]  [font=\fontsize{0.47em}{0.56em}\selectfont] [align=left] {$\displaystyle \langle q\rangle _{I_{k+1}} =v_{k+1}$};
% Text Node
\draw (460.35,108.9) node [anchor=north west][inner sep=0.75pt]  [font=\large] [align=left] {...};
% Text Node
\draw (464.58,148.52) node [anchor=north west][inner sep=0.75pt]  [font=\fontsize{0.33em}{0.4em}\selectfont] [align=left] {$\displaystyle \langle q\rangle _{I_{k+2}} =v_{k+2}$};

\end{tikzpicture}
\caption{Values and averages of $q$}
\label{fig:placeholder}
\end{figure}
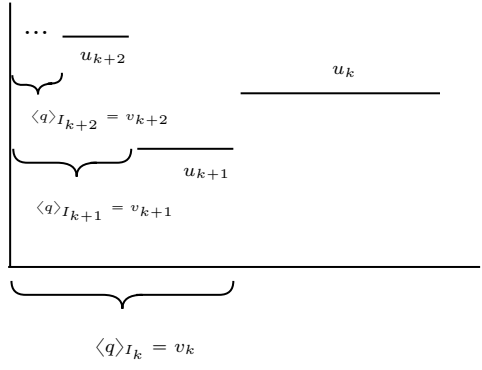

\begin{proposition} The function $q$ satisfies the following properties: 

\begin{enumerate} 
    \item \label{q integ} $q$ is integrable on each $I_k$ and $\LL q\RR_{I_k}=v_k$.
    \item \label{q not in bmo} We have that $\sup_{I\in \mathscr{D}}|\LL q\RR_I-\LL q\RR_{I^s}|=\infty$.
    \item \label{p carl}For every $1<p<\infty $ we have $\sup_{I\in \mathscr{D}}\frac{1}{\mu(I)}\int_I|q-\LL q\RR_I|^pd\mu<\infty$.
    \item \label{boundedness of hilb} For every $1<p<\infty $ $[\mathcal{H},q]$ is bounded on $L^p(\mu)$.
\end{enumerate}
Hence, $q\in [\BMO]_p(\mu)\setminus \BMO(\mu)$ and $\BMO(\mu)\subsetneq [\BMO]_p(\mu)$.
\end{proposition}

\begin{proof}
    We first show that $q$ is integrable on $I_k$; this holds since
    \begin{align*}
        \int_{I_k}|q|d\mu&=\sum_{i\geq k+1}\int_{I_i^b}|u_i|\,d\mu \\
    &= \sum_{i\geq k+1} \left| v_{i-1} - a_i(-1)^i \log i \right| \left(\prod_{j=1}^{i-1}a_j\right) b_i \\
    &= \sum_{i\geq k+1} \left| \left( v_{i-1} - \frac{(-1)^i \log i}{\sqrt{i}} \right) a_{i-1} \right| \left(\prod_{j=1}^{i-2}a_j\right) b_i \\
    &\sim \sum_{i\geq k+1} \left| v_{i-1} a_{i-1} - \frac{(-1)^i \log i}{i} \right| \frac{b_i}{b_{i-1}} \left(\prod_{j=1}^{i-2}a_j\right) b_{i-1} < \infty.
    \end{align*}
    The last sum is convergent because the series of $\{(\prod_{j=1}^{i-1}a_j)b_i\}_{i}$ is convergent, $v_{i-1}a_{i-1}$ is bounded, and $b_i/b_{i-1}$ is roughly equal to 1 for large $i$. To prove \ref{q integ}, using \eqref{prop of u,v} we compute similarly 
    \begin{align*}
        \int_{I_k}qd\mu
        &=\sum_{i\geq k+1}u_i(\prod^{i-1}_{j=1}a_j)b_i\\
    &=\sum_{i\geq k+1}(v_{i-1}-v_ia_i)(\prod^{i-1}_{j=1}a_j)\\
    &=\lim_{n\rightarrow\infty}(v_k\prod^{k}_{j=1}a_j-v_{n}\prod^{n}_{j=1}a_j)\\
    &=v_k\prod^{k}_{j=1}a_j\\
    &=v_k\mu(I_k).
    \end{align*}
    Notice that in the last equality we used again the boundedness of $|v_na_n|$ and $\lim_{n\rightarrow\infty}\prod^{n-1}_{j=1}a_j=0$. \\
   To prove \ref{q not in bmo}, notice that if $I=I_k$, then $I^s=I_k^b$ and using \eqref{prop of u,v}
    \begin{align*}
        \sup_k|\LL q\RR_{I_k}-\LL q\RR_{I_k^b}|&=\sup_k|v_k-u_k|=\sup_k \log(k)=\infty.
    \end{align*}
    We again prove \ref{p carl} through a case by case analysis.
    \begin{enumerate}[label=(\roman*)]
\item Assume $I\subset [0,1)$ and $I \neq I_k$ for every $k \geq 1$. Then $I\subset I_j^b$ for some $j$ and as $q$ is constant on $I_j^b$,
\[\frac{1}{\mu(I)}\int_I|q-\LL q\RR_I|^pd\mu=0.\]
Now consider $|I|<1$ and $I=I_k$ for some $k\geq 1$. \begin{comment}Let $J$ be a dyadic subinterval of $I_k$. If $J\ne I_{k+j}$ for some $j$, then $q$ is constant on $J$ and thus
\begin{align*}
    \sum_{J\subset I_k}|\LL q,h_J\RR|^2&=\sum_{I_{k+j}, j\geq 0}m(I_{k+j})|\LL q\RR_{(I_{k+j})_-}-\LL q\RR_{(I_{k+j})_+}|^2\\
    &=\sum_{j\geq 0}a_{k+j+1}b_{k+j+1}\mu(I_{k+j})|v_{k+j+1}-u_{k+j+1}|^2\\
    &=\sum_{j\geq 0}a_{k+j+1}b_{k+j+1}\prod_{i=1}^{k+j}a_i\sqrt{k+j+1}\\
    &=\sum_{j\geq 0}b_{k+j+1}\prod_{i=1}^{k+j}a_i\\
    &=\prod_{j=1}^ka_j \leq 1.
\end{align*}
We now prove the $p$-Carleson condition for $I_k$ and $p \neq 2$.\end{comment}
Since $\LL q\RR_{I_k}=v_k$, the intervals $I^b_k$ partition $[0,1)$ and $q$ is constant on each of these pieces, then
    \[ \int_{I_k}|q - v_k|^p d\mu = \sum_{j=1}^\infty \int_{I_{k+j}^b} |u_{k+j} - v_k|^p d\mu. \]
    Then using the values of $\mu(I_k)$ and $\mu(I^b_{k+j})$, \begin{align*}
        \frac{1}{\mu(I_k)}\int_{I_k}|q - v_k|^p d\mu &= \frac{1}{\mu(I_k)} \sum_{j=1}^\infty |u_{k+j}-v_k|^p \mu(I_{k+j}^b) \\
        &= \frac{1}{\prod_{i=1}^k a_i} \sum_{j=1}^\infty |u_{k+j}-v_k|^p  \left(b_{k+j} \prod_{i=1}^{k+j-1} a_i\right) \\
        &=\sum_{j=1}^\infty |u_{k+j}-v_k|^p  \, b_{k+j} \left(\prod_{i=k+1}^{k+j-1} a_i\right).
    \end{align*}
    In other words, we need to prove that for fixed $1<p<\infty$ 
\begin{equation*}
 F(k)=\sum_{j=1}^\infty |u_{k+j}-v_k|^p\, b_{k+j}\prod_{i=k+1}^{k+j-1} a_i    
\end{equation*}
is uniformly bounded in $k$ for $k \geq 1$.
We split the difference as
\[
 u_{k+j}-v_k = S(k,j)-R(k,j),
\]
\[ S(k,j)=\sum_{i=k+1}^{k+j-1}(-1)^i\log i,\qquad R(k,j)=\sum_{i=k+1}^{k+j}a_i(-1)^i\log i. \]
Notice that as \(| u_{k+j}-v_k |^p \lesssim_p |S(k,j)|^p+ |R(k,j)|^p\), $R(k,j)$ can be controlled by $S(k,j) + O(1)$ and $|S(k,j)| \lesssim \log(k+j)$ for $j$ big enough. By isolating the first term in the sum, it now suffices to control 
\[|u_{k+1}-v_k|^pb_{k+1} + \sum_{j=2}^\infty |\log(k+j)|^p\, b_{k+j}\prod_{i=k+1}^{k+j-1} a_i. \]
Since $|u_{k+1}-v_k|^p=\log(k)^p k^{-p/2}$ is uniformly bounded in $k$ and $b_{k+j} \leq 1$, we can reduce to study the sum for $j \geq 2$. We then argue that 
\begin{align*}
  \sum_{j=2}^\infty |\log(k+j)|^p\, b_{k+j}\prod_{i=k+1}^{k+j-1} a_i  & \leq \sum_{j=2}^\infty |\log(k+j)|^p\, \prod_{i=k+1}^{k+j-1} a_i \\ & \leq \sum_{j=2}^\infty |\log(k+j)|^p\,(k+1)^{-(j-1)/2}
\end{align*}
where we used that $a_i \leq (k+1)^{-1/2}$ for every $i \geq k+1.$ The last series converges as a consequence of the ratio test whenever $k \geq 1$, so $\sup_{k \in \N} F(k) <\infty$.
\item Now assume $|I|\geq 1$. Recall that $q$ is periodic with period 1. Also recall that $\mu([0,1))=1$ and thus $\mu(I)=|I|=m$ for some positive integer $m$. These two conditions ensure that $\LL q\RR_I=\LL q\RR_{[0,1)}$. The calculation above for $I_k$ clearly also works similarly when $k=0$, so that
\[\frac{1}{\mu(I)}\int_I|q-\LL q\RR_I|^pd\mu=\frac{m\int_0^1|q-\LL q\RR_{[0,1)}|^pd\mu}{m}<\infty.\]
\end{enumerate}
We conclude the proof by showing $\sup_I|c_I(q)-c_{I^s}(q)|<\infty$ and consequently \ref{boundedness of hilb}.
\begin{enumerate}[label=(\roman*)]
    \item If $|I|\geq 1$, then $\mu(I_-)=\mu(I_+)$ because $\mu([0,\frac{1}{2}
    ))=\frac{1}{2}$ and $\mu([0,1))=1$. Consequently,
    \[c_I(q)-c_{I^s}(q)=\LL q\RR_I-\LL q\RR_{I^s}=0.\]
    \item Assume that $I\subset [0,1)$. If $\widehat{I}\subset I_k^b$ for some $k\geq 1$, then as $q$ is constant on $I_k^b$,
    \[c_I(q)-c_{I^s}(q)=\LL q\RR_I-\LL q\RR_{I^s}=0.\]
    We are left with $I=I_k$ or $I=I_k^b$ and, by symmetry, we can assume that $I=I_k$. On $I^s=I_k^b$, $q$ is constant. By the definition of $v_k$ and $u_k$, we have
    \begin{align*}
        c_I(q)-c_{I^s}(q)&=(\LL q\RR_{I_+}-\LL q\RR_{I_-})\frac{\mu(I_-)-\mu(I_+)}{\mu(I)}+\LL q\RR_I-\LL q\RR_{I^s}\\
        &\approx v_{k+1}-u_{k+1}+v_k-u_k\\
        &=(-1)^{k}\log \bigg(\frac{k}{k+1} \bigg).        
    \end{align*}
    Hence \(\sup_k|c_{I_k}(q)-c_{I^b_k}(q)|<\infty\) and this concludes the proof.
\end{enumerate}
\end{proof}

 \subsection{Final remarks and open questions} 
 We comment on some potential areas of future investigation inspired by the results and techniques developed in this paper.
 \begin{enumerate}
 \item The $p$-dependent characterization of commutator symbols suggests that similar hierarchies might exist for other operators or symbols in nonhomogeneous settings. In particular, the precise role the parameter $p$ plays in characterizing both the compactness of commutators on $L^p(\mu)$, and two-weight inequalities of the form $L^p(\mu) \rightarrow L^p(\lambda)$, merit further investigation. One would expect these spaces to be non-homogeneous, $p$-dependent analogs of VMO and Bloom-type BMO spaces, respectively, but the classical proofs will break down in the non-homogeneous setting. Nevertheless, powerful tools developed in this paper will likely help characterize these subtle spaces. 
 \item The ingredients in the sparse domination proof may be broadly applicable to other operators or areas of interest in the dyadic non-doubling setting, including \emph{multilinear} martingale transforms, Haar shifts, paraproducts, commutators, and other dyadic operators. Once again, the classical methods will be insufficient, and one will have to discover the appropriate analog of the non-standard sparse forms in the multilinear setting, which poses an interesting but feasible challenge.
 \item Endpoint estimates for Haar shifts can likely be sharpened via a similar strategy used in \cite{BonamiJXYZ}. The class of operators considered there merely satisfy $T:H^1(\mu) \to L^1(\mu)$, where $H^1$ is the martingale Hardy space, while it was proved in \cite{CAW2025} that Haar shifts obey the stronger bound $T:H^1(\mu) \to H^1(\mu)$ under the balanced assumption.
 Furthermore, the characterization of the pre-duals of the spaces $[\BMO]_p(\mu)$ remains mysterious. We know from simple containment relationships that if $X^*= [\BMO]_2(\mu)$ for example, then $\mathrm{h}^1(\mu) \subsetneq X \subsetneq H^1(\mu)$, where $\mathrm{h}^1(\mu)$ is a Hardy space defined using the conditional square function. It would be interesting to characterize $X$ precisely and explore possible connections to the space $H^1_b$. 

 \item The Petermichl shift $\mathbb{S}$ represents a competing dyadic model of the classical Hilbert transform. The characterization of bounds for commutators of $[\mathbb{S},b]$ remains open.
\end{enumerate}

\bibliographystyle{alpha}
\bibliography{Paraproducts-Sources}
\end{document}